\nonstopmode
\documentclass[reqno,10pt]{amsart}
\usepackage{latexsym}
\usepackage{fancyhdr}
\usepackage{amsmath, amssymb, amsthm}
\usepackage[ansinew]{inputenc}
\usepackage[all]{xy}
\newdir{ >}{{}*!/-10pt/@{>}}
\newdir^{ (}{{}*!/-5pt/@^{(}}
\newdir_{ (}{{}*!/-5pt/@_{(}}

\usepackage{pdflscape}
\usepackage{longtable}
\usepackage{rotating}
\usepackage{verbatim}
\usepackage{hyperref}
\usepackage{subfigure}
\usepackage{mathrsfs}
\usepackage{tensor}

\usepackage{graphicx}
\usepackage{mdwlist}
\usepackage{etoolbox}
\usepackage{todonotes}
\usepackage{esint}
\usepackage{mathtools}
\usepackage{bm}
\usepackage{stmaryrd}
\usepackage{cleveref}
\usepackage{accents}

\usepackage{thmtools, thm-restate} 
\declaretheorem[numberwithin=section]{theorem}

\newcounter{mnotecount}[section]

\newcommand{\rmnote}[1]{}

\theoremstyle{plain}
\newtheorem*{lemma*}{Lemma}
\newtheorem{lemma}[theorem]{Lemma}
\newtheorem*{theorem*}{Theorem}
\newtheorem*{result*}{Result}

\newtheorem*{proposition*}{Proposition}
\newtheorem{proposition}[theorem]{Proposition}
\newtheorem*{corollary*}{Corollary}
\newtheorem{corollary}[theorem]{Corollary}
\newtheorem*{claim*}{Claim}

\theoremstyle{definition}
\newtheorem*{assumption*}{Assumption}

\newtheorem*{definition*}{Definition}
\newtheorem{definition}[theorem]{Definition}
\newtheorem*{convention*}{Convention}
\newtheorem{convention}[theorem]{Convention}
\newtheorem{setup}[theorem]{Setup}
\newtheorem*{setup*}{Setup}
\newtheorem*{example*}{Example}
\newtheorem{example}[theorem]{Example}

\newtheorem*{algorithm*}{Algorithm}
\newtheorem*{remark*}{Remark}
\newtheorem{remark}[theorem]{Remark}

\numberwithin{equation}{section}

\def\Hoeld{\on{H\ddot{o}ld}}
\def\Lip{\on{Lip}}
\def\Var{\on{Var}}

\def\al{\alpha}
\def\be{\beta}
\def\ga{\gamma}
\def\de{\delta}
\def\ep{\epsilon}

\def\th{\theta}

\def\la{\lambda}

\def\si{\sigma}

\def\ph{\phi}
\def\vh{\varphi}

\def\ps{\psi}
\def\om{\omega}

\def\La{\Lambda}

\def\Om{\Omega}

\def\B{\mathbb{B}}
\def\C{\mathbb{C}}

\def\N{\mathbb{N}}
\def\Q{\mathbb{Q}}
\def\R{\mathbb{R}}
\def\S{\mathbb{S}}
\def\Z{\mathbb{Z}}

\def\cC{\mathcal{C}}
\def\cD{\mathcal{D}}

\def\cH{\mathcal{H}}

\def\cL{\mathcal{L}}

\def\cU{\mathcal{U}}

\def\sD{\mathscr{D}}

\def\sI{\mathscr{I}}

\def\sK{\mathscr{K}}

\def\sfC{\mathsf{C}}

\def\sfE{\mathsf{E}}

\def\p{\partial}
\def\Var{\on{Var}}
\def\ind{\,\mathbf{1}}

\def\<{\langle}
\def\>{\rangle}
\renewcommand{\o}{\circ}

\def\loc{\on{loc}}
\def\div{\on{div}}
\def\restr{\,\llcorner\,}
\def\sgn{\on{sgn}}

\def\tD{\widetilde D}

\def\kk{{\ell}}

\def\cover{\mathcal {CV}}

\newcommand{\ol}{\overline}
\newcommand{\ul}{\underline}

\DeclareMathOperator{\im}{Im}

\let\on=\operatorname
\newcommand{\sr}[1]%
{\ifmmode{}^\dagger\else${}^\dagger$\fi\ifvmode
\vbox to 0pt{\vss
 \hbox to 0pt{\hskip\hsize\hskip1em
 \vbox{\hsize3cm\raggedright\pretolerance10000
 \noindent #1\hfill}\hss}\vss}\else
 \vadjust{\vbox to0pt{\vss%
 \hbox to 0pt{\hskip\hsize\hskip1em%
 \vbox{\hsize3cm\raggedright\pretolerance10000%
 \noindent #1\hfill}\hss}\vss}}\fi%
}

\title[Selections of bounded variation]{Selections of bounded variation\\ for roots of smooth polynomials}

\author[Adam Parusi\'nski and Armin Rainer]
{Adam Parusi\'nski and Armin Rainer}

\address {Adam Parusi\'nski: Universit\'e C\^ote d'Azur,  CNRS,  LJAD, UMR 7351, 06108 Nice, France}

\email{adam.parusinski@univ-cotedazur.fr>}

\address{Armin Rainer: Fakult\"at f\"ur Mathematik, Universit\"at Wien, 
Oskar-Morgenstern-Platz~1, A-1090 Wien, Austria
\& University of Education Lower Austria,
Campus Baden M\"uhlgasse 67, A-2500 Baden, Austria}

\email{armin.rainer@univie.ac.at}

\begin{document}

\begin{abstract}
    We prove that the roots of a smooth monic polynomial with complex-valued 
    coefficients defined on a bounded Lipschitz domain $\Om$ in $\R^m$ 
    admit a parameterization by functions of bounded variation 
    uniformly with respect to the coefficients. 
    This result is best possible in the sense that discontinuities of the roots 
    are in general unavoidable due to monodromy. We show that the discontinuity 
    set can be chosen to be a finite union of smooth hypersurfaces. 
    On its complement the parameterization of the roots is of optimal Sobolev class 
    $W^{1,p}$ for all $1 \le p < \frac{n}{n-1}$, where $n$ is the degree of the 
    polynomial.
    All discontinuities are jump discontinuities.
    For all this we require the coefficients to be of class $C^{k-1,1}(\ol \Om)$, 
    where $k$ is a positive integer depending only on $n$ and $m$. 
    The order of differentiability $k$ is not optimal. 
    However, in the case of radicals, i.e., for the solutions of the equation $Z^r = f$, 
    where $f$ is a complex-valued function and $r\in \R_{>0}$, 
    we obtain optimal uniform bounds.    
\end{abstract}

\thanks{Supported by the Austrian Science Fund (FWF) Grant P~26735-N25, by ANR project LISA (ANR-17-CE40-0023-03),  
and by ERC advanced grant 320845 SCAPDE.}

\date{\today}

\maketitle

\tableofcontents

\section{Introduction}

\subsection{The main results}

Let $\Om \subseteq \R^m$ be an open set and let 
    \begin{equation} \label{polynomials}
      P_a(x)(Z)= P_{a(x)}(Z) = Z^n + \sum_{j=1}^n a_j(x) Z^{n-j}, \quad x \in \Om, 
    \end{equation}
  be a monic polynomial with complex-valued coefficients
  $a=(a_1,\ldots,a_n) : \Om  \to \C^n$. 
The roots of $P_a$ form a multi-valued function $\La : \Om \leadsto \C$.  
If $a$ is of H\"older class $C^{n-1,1}(\ol \Om)$, then $\La$ is of Sobolev class $W^{1,p}(\Om)$, 
for all $1 \le p < \frac{n}{n-1}$, 
in the sense of Almgren \cite{Almgren00} (see also \cite{De-LellisSpadaro11}), 
and this result is sharp.  
This follows from the main result of our recent paper \cite{ParusinskiRainer15}; 
 see \cite[Theorem 6]{ParusinskiRainer15}.

In this paper we study the existence of regular selections and parameterizations of the multi-valued function 
$\La$.  
The point-image $\La(x)$, for $x \in \Om$, is the unordered $n$-tuple consisting of the roots of $P_a(x)$ (with multiplicities).
A \emph{parameterization} of $\La$ is an $n$-tuple $\la = (\la_1,\ldots,\la_n)$ of single-valued functions such that 
$\la(x)$ represents $\La(x)$ for all $x \in \Om$. 
A \emph{selection} of $\La(x)$ is a single-valued function $\mu$ such that $\mu(x) \in \La(x)$ for all $x \in \Om$, 
or equivalently $P_a(x)(\mu(x)) =0$.

The main result of \cite{ParusinskiRainer15} states that any \textbf{continuous} selection of the roots of $P_a$, 
where $a \in C^{n-1,1}(\ol I,\C^n)$ and $I$ is an open bounded interval in $\R$, 
is of class $W^{1,p}(I)$, for all $1 \le p < \frac{n}{n-1}$, uniformly with respect to coefficients.
This result is optimal. 
It is not hard to see that in this one-dimensional case there always exist continuous parameterizations of the roots
(e.g.\ \cite[Ch.~II Theorem~5.2]{Kato76}).  

As a consequence, any continuous selection $\mu : V \to \C$ of a root of $P_a$, 
where $a \in C^{n-1,1}(\ol \Om,\C^n)$, $\Om$ is a Lipschitz domain,  
and $V \subseteq \Om$ is an open subset, 
is of class $W^{1,p}(V)$, for all $1 \le p < \frac{n}{n-1}$ (see \Cref{optimal2}). But, 
for dimension $m\ge 2$, monodromy in general prevents the existence of continuous selections
of roots on $\Om$. 
So it is natural to ask:
\begin{quote}
	\emph{Can the roots of a polynomial \eqref{polynomials} with coefficients in a differentiability class of 
	sufficiently high order 
	be represented by functions of bounded variation?}	
\end{quote}

Functions of bounded variation ($BV$) are integrable function whose distributional derivative is a vector-valued finite 
Radon measure. They form an algebra of discontinuous functions. 
Due to their ability to deal with discontinuities they are widely used in the applied sciences, 
see e.g.\ \cite{KhudyaevVol'pert85}.

Our main result gives a positive answer to the above question:

\begin{theorem} \label{thm:main}
	For all integers $n,m\ge 2$ there exists an integer $k=k(n,m) \ge \max(n,m)$ such that the following holds.
	Let $\Om \subseteq \R^m$ be a bounded Lipschitz domain and let \eqref{polynomials} 
	be a monic polynomial with complex-valued coefficients
	$a=(a_1,\ldots,a_n)  \in C^{k-1,1}(\ol \Om,\C^n)$. 

	Then the roots of $P_a$ admit a parameterization $\la = (\la_1,\ldots,\la_n)$ by 
	special functions of bounded variation ($SBV$) on $\Om$ such that
	\begin{equation} \label{mainbound}
		\|\la\|_{BV(\Om)} \le C(n,m,\Om)\, \max\big\{1, \| a\|_{L^\infty(\Om)}\big\} \max 
    \big\{1 , \| a\|_{C^{k-1,1}(\ol \Om)}\big\}. 
	\end{equation}
  There is a finite collection of $C^{k-1}$-hypersurfaces $E_j$ in $\Om$ 
 	such that $\la$ is continuous in the complement of $E :=\bigcup_j E_j$. 
 	Any hypersurface $E_j$ is closed in an open subset of $\Om$ 
 	but possibly not in $\Om$ itself. All discontinuities of $\la$ are jump discontinuities. 
\end{theorem}

Note that we do not claim that the discontinuity 
hypersurfaces $E_j$ have finite ($m-1$)-dimensional Hausdorff measure 
($\cH^{m-1}$).  Actually we construct an example, see \Cref{example}, where this is not possible.

If $n = 1$ then the problem has a trivial solution $\la = -a_1$.
If $n\ge 2$ and $m = 1$, then the problem was solved in \cite{ParusinskiRainer15}. 
In both cases the roots admit continuous parameterizations.

A function of bounded variation is called \emph{special} ($SBV$) if the Cantor part of its derivative vanishes; 
for precise definitions and background on $BV$-functions we refer to \Cref{preliminaries}. 

\begin{remark}
   The dependence of $k$ on $m$ stems from the use of Sard's theorem. 
   The dependence of the constant $C$ on $\Om$ originates from the use of Whitney's extension theorem (see \Cref{sec:prep2}) 
   and also from the trivial bound 
  $\|\la\|_{L^1(\Om)} \le |\Om| \,\|\la\|_{L^\infty(\Om)}$. 
       It is well-known that $\max_{1 \le i \le n} |\la_i(x)| \le 2 \max_{1 \le j \le n} |a_j(x)|^{1/j}$ for all $x$ 
       (cf.\ \cite[p.56]{Malgrange67} or \cite[{(8.1.11)}]{RS02}).
\end{remark} 

Together with \Cref{optimal2} of \Cref{appendix} we immediately obtain the following 
supplement.

\begin{corollary}
  The parameterization $\la$ satisfies
  \begin{equation} \label{eq:Sobolevbound}
    \|\la \|_{W^{1,p}(\Om\setminus \ol E)}  \le  C(n,m,p,\Om) \max_{1 \le j \le n} \|a_j\|^{1/j}_{C^{n-1,1}(\overline \Om)}
  \end{equation}
  for all $1 \le p < \frac{n}{n-1}$.
\end{corollary}

For radicals, i.e., solutions of $Z^r = f$, where we allow $r \in \R_{>0}$, we have better bounds:

\begin{theorem}
\label{BVradicals}
  Let $r \in \R_{>0}$ and $m\in \N_{\ge 2}$. 
  Let $k \in \N$ and $\al \in (0,1]$ be such that  
  $k+\al \ge \max\{r,m\}$. 
  Let $\Om \subseteq \R^m$ be a bounded Lipschitz domain.   
  Let $f \in C^{k,\al}(\overline \Om)$. 
  
  Then there exists a solution $\la \in SBV(\Om)$ of the equation $Z^r = f$ such that 
  \begin{align} \label{eq:uniformbound}
    \|\la\|_{BV(\Om)} \le C(m,k,\al,\Om)\, \|f\|_{C^{k,\al}(\ol \Om)}^{1/r}.
  \end{align} 
  There is a $C^k$-hypersurface $E \subseteq \Om$  
  (possibly empty) such that $\la$ is continuous on $\Om \setminus E$ and satisfies 
  $\nabla \la \in L^p_w(\Om \setminus \overline E)$ for 
  \[
  p = \begin{cases}
    \frac{r}{r-1} & \text{ if } r>1, \\
    \infty & \text{ if } r \le 1 \text{ and } r^{-1} \in \N, \\
    \frac{\{r^{-1}\}^{-1}}{\{r^{-1}\}^{-1}-1} & \text{ if } r < 1 \text{ and } r^{-1} \not\in \N.
  \end{cases}
  \]  
  We have  
  \begin{align} 
    \|\nabla \la\|_{L^p_w(\Om\setminus \ol E)} \le C(m,k,\al,\Om)\, \|f\|_{C^{k,\al}(\overline \Om)}^{1/r},
    \intertext{and}
    \int_{E}  |f|^{1/r} \, d \cH^{m-1} \le C(m,k,\al,\Om)\, \|f\|_{C^{k,\al}(\ol \Om)}^{1/r}.
  \end{align} 
  All discontinuities are jump discontinuities. 
\end{theorem}

Here $L^p_w(V)$ denotes the weak Lebesgue space equipped with the quasinorm $\|\cdot\|_{p,w,V}$ 
and $\{x\}$ is the fractional part of $x \in \R$. 
This result is optimal as follows from \Cref{rem:optimality}: in general, $\nabla \la \not \in L^p$, 
even if $\la$ is continuous and $f$ is real analytic, and 
$\la$ need not have bounded variation if $f$ is only of class $C^{\ell,\be}$ whenever $\ell+\be <r$.

\begin{remark}
	In the case that $r = n$ is an integer, 
	a complete parameterization of the roots of $Z^n =f$ is provided by 
$\th^k \la$, $k =0, \ldots, n-1$, where $\th = e^{2 \pi i/n}$. 
\end{remark}

\subsection{Background}

The problem of determining the optimal regularity of the roots of 
univariate monic polynomials whose coefficients depend smoothly on parameters has a long and rich history.
Its systematic investigation probably started with Rellich's work on the perturbation theory of 
symmetric operators in the 1930s \cite{Rellich37I,Rellich37II,Rellich39III,Rellich39III,Rellich40IV,Rellich42V}, see 
also his monograph \cite{Rellich69}. This line of research culminated with Kato's monograph \cite{Kato76}. 
But the regularity problem of the eigenvalues of symmetric, Hermitian, and even normal matrices/operators 
behaves much better in many aspects than the related problem of choosing regular roots of smooth families of 
polynomials; see \cite{RainerN} for a survey of the known results.

The regularity of square roots of non-negative smooth functions was first studied by Glaeser \cite{Glaeser63R}. 
The general case of \emph{hyperbolic} polynomials (i.e.\ all roots are real) plays a crucial role for the 
Cauchy problem for hyperbolic partial differential equations with multiple characteristics. The central result 
is this connection is Bronshtein's theorem \cite{Bronshtein79}:
every continuous choice of the roots of a hyperbolic monic polynomial of degree $n$ with $C^{n-1,1}$-coefficients 
is locally Lipschitz. Note that there always is a continuous parameterization of the roots in this case, e.g., 
by ordering them increasingly. Variations on this fundamental result (and its proof) appeared in 
\cite{Mandai85}, \cite{Wakabayashi86}, \cite{AKLM98}, \cite{KLM04}, 
    \cite{BBCP06}, \cite{BonyColombiniPernazza06}, \cite{Tarama06}, \cite{BonyColombiniPernazza10}, 
    \cite{ColombiniOrruPernazza12}, 
    \cite{ParusinskiRainerHyp}.

The \emph{complex} (i.e.\ not necessarily hyperbolic) counterpart, which is the problem at the center of this paper,
was considered for the first time (for radicals) by Colombini, Jannelli, and Spagnolo \cite{CJS83}. 
Motivated by the analysis of certain systems of pseudo-differential 
equations Spagnolo \cite{Spagnolo00} asked if the roots of a smooth \emph{curve} of monic polynomials admit a 
parameterization by locally absolutely continuous functions.
This conjecture was proved in our papers \cite{ParusinskiRainerAC} and \cite{ParusinskiRainer15} 
which are based on the solution for radicals due to Ghisi and Gobbino \cite{GhisiGobbino13}. 
The optimal Sobolev regularity of the roots, which was already mentioned above, was established in 
\cite{ParusinskiRainer15} by elementary methods. 
Its precursor \cite{ParusinskiRainerAC} in which the optimal bounds 
were still missing
was based on Hironaka's resolution of singularities. We wish to mention that absolute 
continuity of the roots was also shown in \cite{ColombiniOrruPernazza17} by different methods; 
for polynomials of degree $n \le 3$ it is due to \cite{Spagnolo99}. 
As already pointed out, a curve $I \ni t  \mapsto P_a(t)$, where $I \subseteq \R$ is an interval, always
admits a continuous parameterization of its roots.  
Further contributions with partial solutions appeared in  
\cite{Tarama00}, \cite{CL03}, \cite{CC04}, \cite{RainerAC}, \cite{RainerQA}, \cite{RainerOmin}, \cite{RainerFin}. 

The results of this paper complete this analysis and solve Open Problem 3 posed in \cite{ParusinskiRainer15}.

\subsection{Idea of the proof}

The main difficulty of the problem is to make a \emph{good} choice of the discontinuity set of the roots.
On the complement of the discontinuity set the roots are of optimal Sobolev class $W^{1,p}$, 
for all $1 \le p < \frac{n}{n-1}$,  
by the result of \cite{ParusinskiRainer15}, see also \Cref{optimal2}. 
In general, the discontinuity set has infinite codimension one Hausdorff measure, see \Cref{example}. 
Thus, in order to have bounded variation it is crucial that the jump height of a selection of a root is 
integrable (with respect to $\cH^{m-1}$) along its discontinuity set.

The proof of \Cref{thm:main} is based on the radical case solved in \Cref{BVradicals} and on formulas for the 
roots of the universal polynomial $P_a$, $a \in \C^n$, which were found in \cite{ParusinskiRainerAC}. 
Interestingly, the method of \cite{ParusinskiRainerAC} seems to be better suited for the control of 
the discontinuities and integrability along them than a more elementary method of \cite{ParusinskiRainer15}.

\subsubsection{The radical case} 
Consider the equation $Z^n = f$, where $f$ is a smooth complex-valued function. 
We choose the discontinuities of the solutions along the preimage of a regular value of the \emph{sign} 
$\sgn(f) = f/|f| : \Om \setminus f^{-1}(0) \to \S^1$ of $f$. 
The result of Ghisi and Gobbino \cite{GhisiGobbino13} which we recall in \Cref{GhisiGobbino2} 
together with the coarea formula and Sard's theorem empowers us to show that 
\[
\int_{\sgn(f)^{-1}(y)} |f|^{1/n} \, d\cH^{m-1} < \infty \quad \text{ for $\cH^1$-a.e.\ } y \in \S^1.
\]
It is then not hard to complete the proof of \Cref{BVradicals} using the observation that any parameterization 
of the solutions of $Z^n  =f$ is continuous (and zero) on the zero set of $f$, and hence 
a $BV$-parameterization on $\Om \setminus f^{-1}(0)$ extends to a $BV$-parameterization on $\Om$ 
with unchanged total variation.

\subsubsection{The general case} 
The formulas for the roots of the universal polynomial $P_a$, $a \in \C^n$, which we recall in detail in 
\Cref{sec:formulas},
express the roots as finite sums of functions analytic in radicals of local coordinates on a resolution space 
(a blowing up of $\C^n$). 
We choose parameterizations of the involved radicals, using \Cref{BVradicals},
and show that in this way we obtain $SBV$-parameterizations of these summands.
But then a new difficulty arises which comes from the fact that these summands are defined only locally on the resolution space. 
(Actually, they cannot be defined neither globally nor canonically.) 
We solve this problem by cutting and pasting these locally defined summands
which introduces new discontinuities.
In order to stay in the class $SBV$ we must ascertain integrability of the new jumps along these discontinuities. 
This is again based on a consequence of Ghisi and Gobbino's result for radicals and the coarea formula, 
see \Cref{cutglue}.

\subsection{Open problems}

The uniform bound in \eqref{mainbound} is not scale-invariant. 
Nor is the degree of differentiability $k$ sharp (in contrast, 
it is sharp in \eqref{eq:Sobolevbound} and \eqref{eq:uniformbound}). 
These deficiencies stem from the method of proof involving 
resolution of singularities. 
\emph{Are there better bounds with lower differentiability requirements?}

The method of our proof is local. This forces us to deal with the global monodromy 
by cutting and pasting the local choices of the roots. 
It introduces additional discontinuities some of which are perhaps not necessary. 
\emph{It would be interesting to have a global understanding of the monodromy and the discontinuities it necessitates.}

\subsection{Structure of the paper}

In \Cref{sec:difficulty} we investigate the discontinuities of radicals caused by monodromy and show in \Cref{example} that their codimension one Hausdorff measure is in general infinite. 
The main analytic ingredient for the proofs of \Cref{thm:main} and \Cref{BVradicals} which allows us to control the integrability of the jump height of the roots along their discontinuity sets is developed in \Cref{sec:analytic}; it is presented in greater generality, see \Cref{prop:level}, since it might be of independent interest. 
In \Cref{preliminaries} we recall the required background on functions of bounded variation. 
The proof of \Cref{BVradicals} is completed in \Cref{proof}. 
The remaining sections are dedicated to the proof of \Cref{thm:main}.
We recall the formulas for the roots of the universal polynomial in \Cref{sec:formulas}. 
In \Cref{generalproof} we prove \Cref{thm:main} modulo the local \Cref{prop:onechart}  
which is then shown in \Cref{localproof}.
In \Cref{appendix} we refine a result from \cite{ParusinskiRainer15} on the Sobolev regularity of 
continuous roots.

\subsection*{Notation}

We use 
$\N := \{0,1,2,\ldots\}$, $\N_{>m} := \{n \in \N : n > m\}$, $\N_{\ge m} := \N_{>m} \cup \{m\}$, $\N_+ := \N_{>0}$, and 
similarly, $\R_+$, $\R_{>t}$, etc.
For $r \in \R_+$ let $\lfloor r \rfloor$ be its integer part and $\{r\} := r - \lfloor r \rfloor$ its 
fractional part.

For $a \in \R^\ell$ and $b \in \R^m$, we denote by $a \otimes b$ the $\ell \times m$ matrix $(a_ib_j)_{i=1,j=1}^{\ell,m}$.
By $B_r(x) = \{y \in \R^m : |x-y|<r\}$ we mean the open ball with center $x$ and radius $r$. 
The open unit ball in $\C^n$ is denoted by $\B$, the unit sphere in $\R^m$ by $\S^{m-1}$. 
By $V(\sI)$ we denote the zero set of an ideal $\sI$.

For a positive measure $\mu$ and a $\mu$-measurable set $E$, let $\mu \restr E$ denote the restriction of $\mu$ to $E$, i.e., 
$(\mu \restr E)(F) = \mu(F \cap E)$.
The $m$-dimensional Lebesgue measure is denoted by $\cL^m$; we also use $\cL^m (E) = |E|$ and $d\cL^m = dx$. 
We write $\fint_E f \,dx$ for the average $|E|^{-1} \int_E f \,dx$. 
By $\cH^d$ we mean the $d$-dimensional Hausdorff measure.

For a mapping $f : X \to Y$ between metric spaces $X$, $Y$ and $\al \in (0,1]$, we set 
\[
  \Hoeld_{\al,X} (f) := \sup_{\substack{x_1,x_2 \in X\\x_1\ne x_2}}  \frac{d(f(x_1),f(x_2))}{d(x_1,x_2)^\al} 
  \quad \text{ and }\quad   \Lip_X(f) := \Hoeld_{1,X} (f). 
\]
Then $f$ is said to be $\al$-H\"older (or Lipschitz) if $\Hoeld_{\al,X} (f) < \infty$ (or $\Lip_X (f)< \infty$). 

Let $\Om \subseteq \R^m$ be open.  
We denote by $C^{k,\al}(\Om)$ the space of complex-valued $C^k$-functions on $\Om$ such that $\p^\ga f$ is locally 
$\al$-H\"older for all $|\ga|=k$. If $\Om$ is bounded, then $C^{k,\al}(\ol \Om)$ is the subspace of functions $f$ such that 
$\p^\ga f$ has a continuous extension to $\ol \Om$ for all $0 \le |\ga| \le k$ and $\Hoeld_{\al,\Om} (\p^\ga f) < \infty$ 
for all $|\ga| = k$. Then $C^{k,\al}(\ol \Om)$ is a Banach space with the norm
\[
  \|f\|_{C^{k,\al}(\ol \Om)} := \sup_{|\ga| \le k, \, x \in \Om} | \p^\ga f(x)| + \sup_{|\ga| =k} \Hoeld_{\al,\Om} (\p^\ga f). 
\]
This norm makes also sense for $\Om = \R^m$. 
If $f = (f_1,\ldots,f_n) : \Om \to \C^n$ is a vector-valued function, then we put 
\[
  \|f\|_{C^{k,\al}(\ol \Om)} := \max_{1 \le i \le n} \|f_i\|_{C^{k,\al}(\ol \Om)}. 
\] 
For real-valued functions $f$ and $g$ we write $f \lesssim g$ if $f \le C g$ for some universal constant $C$.

\section{Discontinuity due to monodromy} \label{sec:difficulty}

In this section we investigate the first difficulty of the problem: 
discontinuities of the roots caused by the local monodromy.  
This difficulty is already present for radicals of smooth functions. 
Thus we concentrate on the solutions of 
\begin{equation} \label{radicalequation}
    Z^r = f,
  \end{equation}  
where $r>1$ is a real number 
and $f$ is a complex-valued smooth not identically equal to zero function defined in some open subset $\Om$ of $\R^m$. Let us explain what we mean by a solution of \eqref{radicalequation} for $r\not \in \Z$.  
Firstly, any solution should vanish on the zero set of $f$. 
On the set $\Om_0 := \Om \setminus f^{-1}(0)$, \Cref {radicalequation} can be given an equivalent form 
$Z = \exp(r^{-1} \log f)$.  Thus, by definition, a continuous function $\lambda (x) $ is a solution of 
\eqref {radicalequation} if there is a branch of logarithm $\log f$ such that $\la = \exp(r^{-1} \log f)$.  Here by $\log f$ we mean a function defined on $\Om_0$, not necessarily continuous, such that $\exp (\log f) = f$.  
For $r$ irrational, if $\exp(r^{-1} \log f)$ is continuous 
then so is $\log f$.  Note also that, if $\la$ is a solution of 
\eqref{radicalequation} then so are $\la \exp(2\pi ik/r)$, $k\in \Z$.  
In particular, for $r$ irrational, if there is a continuous  $r$-th root of $f$, 
 then there are infinitely many of them.
  If $r = a/b$ with $a,b$ being two relatively prime integers,
then $Z = \exp(r^{-1} \log f)$ can be written equivalently as $Z^a= f^b$.

Consider the \emph{sign} function of $f$ defined by 
\begin{equation} \label{sgnf}
  \sgn(f) : \Om_0 \to \S^1, \quad \sgn(f) := \frac{f}{|f|}. 
\end{equation}
The existence of a continuous selection of $f^{1/r}$   
 depends on the image $\im (\pi_1 (\sgn f))$ of the induced homomorphism  of the fundamental groups $
 \pi_1 (\sgn f):\pi_1(\Om_0) \to \pi_1 (\S ^1) = \Z$ (here we suppose for simplicity that $\Om_0 $ is connected).  More precisely, a continuous selection of $f^{1/r}$  exists  
 if and only if 
\begin{enumerate}
  \item  
	$\im (\pi_1 (\sgn f)) \subseteq n\Z$, if $r = n\ge 1$ is an integer, 
	\item
  $\im (\pi_1 (\sgn f)) \subseteq a\Z$, if $r = a/b$ with $a,b$ being two relatively prime integers.  
  (Indeed, the equation $Z^a =f^b$ admits a continuous solution if and only if
   the equation $Z^a =f$ has one.  For instance, if $\la^a= f^b$ then $(\la ^l f^k)^a=f$, for the integers $k,l$ such that $ak+bl=1$.)
  \item
$\im (\pi_1 (\sgn f)) =0$, that is $\pi_1(\sgn f)$ is zero, if $r$ is irrational.
 \end{enumerate}

If the above stated conditions are not satisfied, then every solution of \eqref{radicalequation} has to be discontinuous. 
We will see in \Cref{discont} that in this case the discontinuity set can be chosen to be a smooth hypersurface provided that 
$f$ is differentiable of sufficiently high order. 
But, in general, as shows \Cref{example}, the $\cH^{m-1}$-measure of the discontinuity set is infinite!
Off its discontinuity set every solution of \eqref{radicalequation} is of Sobolev class $W^{1,p}$ for all 
$1\le p < \frac{r}{r-1}$ which follows from a result of Ghisi and Gobbino \cite{GhisiGobbino13} which we recall in 
\Cref{sec:contrad}.

We shall use a version of Sard's theorem which we recall for convenience.

\subsection{Sard's theorem}

The following extension of Sard's theorem to Sobolev spaces is due to \cite{Pascale01}; see also
\cite{Figalli08} for a different proof.

\begin{theorem}[Sard's theorem] \label{Sard}
  Let $\Om \subseteq \R^m$ be open.
  Let $f : \Om \to \R^\ell$ be a $W^{m-\ell +1,p}_{\on{loc}}$-function, where $p>m \ge \ell$.
  Then the set of critical values of $f$ has $\cL^\ell$-measure zero.   
\end{theorem}

In the case $m=\ell$ the result follows from a theorem of Varberg and the fact that a $W^{1,p}$  
self-mapping of $\R^m$ satisfies the Luzin N-property if $p>m$; cf.\ the discussion in \cite[Section 5]{Pascale01}.

In particular, the conclusion holds for each $f \in C^{m-\ell,1}(\Om,\R^\ell)$, where $m \ge \ell$.
See also \cite{Bates93} and \cite{Norton94}.

\subsection{On the discontinuity set of radicals}

The roots of the equation $Z^n = x$, $x \in \C$, $n \in \N_{\ge2}$ do not admit a continuous choice in any neighborhood 
of the origin. However, the roots can be chosen continuously on any set $\C \setminus \R_{+} v$, where $v \in \S^1$. 
The next proposition generalizes this fact.

\begin{proposition} \label{discont}
  Let $r \in \R_{>1}$, $m \in \N_{\ge 2}$,  and $k \in \N_{\ge m-1}$.      
  Let $\Om \subseteq \R^m$ be open and let $f \in C^{k,1}(\Om)$. 
  Then there exist a $C^k$-hypersurface $E \subseteq \Om$  
  (possibly empty) and a continuous function 
  $\la : \Om\setminus E \to \C$ such that $\la^r = f$.
\end{proposition}

\begin{proof}
  Let $\Om_0 := \Om \setminus f^{-1}(0)$. 
  Clearly $\sgn(f) : \Om_0 \to \S^1$ is a $C^{k,1}$-mapping.
  
  If $\sgn(f)(\Om_0) \ne \S^1$ define $E := \emptyset$ and choose $v \in \S^1 \setminus \sgn(f)(\Om_0)$. 
  If $\sgn(f)(\Om_0) = \S^1$,
  then there exists a regular value $v \in \S^1$ of $\sgn(f)$, by Sard's theorem (cf.\ \Cref{Sard}), and  
  the set $E := \sgn(f)^{-1}(v)$ is a $C^k$-hypersurface. 
  
  In either case define $\la := f^{1/r} = \exp(r^{-1} \log f)$,   where the logarithm is understood to have its branch cut along the ray $\R_+ v$. 
  Then $\la$ is continuous (even $C^{k,1}$) on $\Om_0\setminus E$ and satisfies $\la^r = f$. 
  Clearly, $\la$ extends continuously by $0$ to
  the zero set of $f$. 
\end{proof}

\begin{remark}
  Note that $E$ is closed in $\Om_0$ but not necessarily in 
  $\Om$.
\end{remark}

The following example shows that in general 
we cannot choose the radicals of a smooth function with compact support in $\R^m$ 
in such a way that its discontinuity set 
has finite $\cH^{m-1}$-measure. 

\begin{example} \label{example}
  There exists a $C^\infty$-function $f : \R^2 \to \C$ with compact support such that the discontinuity set $S_\la$ of any 
  function $\la : \R^2 \to \C$ with $\la^2=f$ satisfies $\cH^1(S_\la) = \infty$.  
\end{example}

\begin{proof}
  Consider a collection $\cD = \{D_k\}_{k=1}^\infty$ of pairwise disjoint open disks 
  $D_k=\{z \in \C : |z-p_k| <\frac{1}{k}\}$. 
  The total area of this collection is
  \[
    |\bigcup_{k=1}^\infty D_k| = \sum_{k=1}^\infty \frac{\pi}{k^2} = \frac{\pi^3}{6}. 
  \]   
  We may assume that the disks $D_k$ are distributed such that $\bigcup_{k=1}^\infty D_k$ is bounded. 
  In fact, the entire collection $\cD$ fits in the rectangle 
  $R = (0,4) \times (0,2)$. 
  To see this let $\cD_n := \{D_k \in \cD : 2^{n-1} \le k < 2^{n}\}$, for $n \in \N_+$.  
  Any disk in $\cD_n$  
  fits in a square of side-length $2^{-n+2}$; there are $2^{n-1}$ such disks. 
  Subdivide the rectangle $R$ by the vertical lines $x=\sum^n_{j=1} 2^{-j+2}$, $n \in \N_+$, which provides 
  a family of open disjoint rectangles $\{R_n\}_{n=1}^\infty$ of dimensions $2^{-n+2} \times 2$. 
  Let us decompose each $R_n$ into a collection $\cC_n$ of $2^{n-1}$ pairwise disjoint squares. 
  By distributing the disks in $\cD_n$ to the squares in $\cC_n$ we achieve 
  $\bigcup_{k=1}^\infty D_k \subseteq R$.

  Let $h : \R \to [0,1]$ be a $C^\infty$-function such that $h(x)=0$ if $x\le 1/4$ and 
  $h(x)=1$ if $x \ge 1$. Then the $C^\infty$-function $h_k : \C \to [0,1]$ given by 
  \[
    h_k(z) := 1 - h(k^2|z-p_k|^2)
  \]
  vanishes outside of $D_k$ and equals $1$ on $D_k' := \{z \in \C : |z-p_k| < \frac{1}{2k}\}$. 
  We claim that the function $f : \C \to \C$ defined by 
  \begin{equation} \label{eq:def1}
    f(z) := \sum_{k=1}^\infty h_k(z) \frac{z-p_k}{2^k}
  \end{equation}
  is $C^\infty$. Indeed, the sum consists of at most one term at any point $z$. 
  Set $c_k(z) := 2^{-k} (z-p_k)$, 
  $H_{\ell} := \sup_{t \in \R, i\le \ell} |h^{(i)}(t)|$, and for 
  $\al \in \N^2$ with $|\al|=\ell$ consider 
  \begin{align*}
    \sup_{z \in \C} |\p^\al (h_k c_k)(z)|  
    &= \sup_{z \in D_k} |\p^\al (h_k c_k)(z)| \\
    &\le \sup_{z \in D_k} \sum_{\be\le \al} \binom{\al}{\be} |\p^\be h_k(z)||\p^{\al-\be} c_k(z)| \\
    &\le  C_\ell H_\ell\, k^{2\ell} 2^{-k}.
  \end{align*}
  The right-hand side is summable and thus the series in \eqref{eq:def1} converges uniformly 
  in each derivative and hence represents an element $f \in C^\infty(\C)$. 

  Let $\la : \C \to \C$ be any function satisfying $\la^2=f$. On the set $D_k'$ we thus have $\la(z)^2 = 2^{-k}(z-p_k)$. 
  For each $r \in (0,(2k)^{-1})$ there exists $q_r \in \{z : |z-p_k| = r\}$ such that $\la$ restricted to $\{z : |z-p_k| = r\}$ 
  is discontinuous at $q_r$.  
  The set $S_k := \{q_r : r \in (0,(2 k)^{-1})\}$ 
  is a subset of the discontinuity set $S_\la$ of $\la$. 
  If $S$ is a subset of $\R^2$ such that $\vh : z \mapsto |z|$ maps $S$ onto $(0,R)$, then 
  \[
     R = \cH^1((0,R)) \le \cH^1(\vh(S)) \le \cH^1(S), 
  \]
  since $\vh$ is Lipschitz with $\Lip_{\R^2}(\vh) =1$.
  Therefore, for all $n \ge 1$,
  \[ 
    \cH^1(S_\la) \ge \cH^1(\bigcup_{k=1}^n S_k) \ge \sum_{k=1}^n  \frac{1}{2 k}, 
  \]
  since any two sets $S_k$ and $S_\ell$ have positive distance if $k\ne \ell$.
  This implies the assertion.
\end{proof}

\subsection{The regularity of continuous radicals} \label{sec:contrad}

The regularity of \emph{continuous} radicals is fully understood thanks the a
result of Ghisi and Gobbino \cite{GhisiGobbino13} which we recall next.

\begin{theorem} 
\label{GhisiGobbino}
  Let $k \in \N_+$, $\al \in (0,1]$, and $r=k+\al$. Let $I \subseteq \R$ be an open bounded interval. Let $\la : I \to \R$ 
  be continuous and assume that there exists $f \in C^{k,\al}(\overline I,\R)$ such that 
  \begin{equation} \label{equationGG}
    |\la|^{r} = |f|.
  \end{equation}
  Let $p$ be defined by $1/p + 1/r =1$.
  Then we have $\la' \in L^p_w(I)$ and 
  \begin{equation} \label{GG}
    \|\la'\|_{p,w,I} \le 
    C(k) \max\Big\{\big(\Hoeld_{\al,I}(f^{(k)})\big)^{1/r}|I|^{1/p}, 
    \|f'\|_{L^\infty(I)}^{1/r}\Big\}, 
  \end{equation}
  where $C(k)$ is a constant that depends only on $k$.
\end{theorem}

For an open set $\Om \subseteq \R^m$, 
 $L^p_w(\Om)$ denotes the \emph{weak Lebesgue space} of functions $f : \Om \to \R$ such that 
\[
  \|f\|_{p,w,\Om} := \sup_{r>0} \Big(r\, \cL^m(\{x \in \Om : |f(x)| > r\})^{1/p} \Big)< \infty.
\]

\begin{remark} \label{rem:optimality}
The result of \Cref{GhisiGobbino} is optimal in the following sense:
\begin{itemize}
  \item In general $\la' \not\in L^{p}(I)$ even if $f \in C^\om(\ol I)$, e.g., $f(t) = t$.
  \item The assumption $f \in C^{k,\al}(\overline I,\R)$ cannot be relaxed to $f \in C^{k,\be}(\overline I,\R)$, 
  for any $\be<\al$. Indeed, 
  there exists a non-negative function $f$ contained in $C^{k,\be}(\overline I,\R) \cap C^\infty(I)$ for all $\be < \al$ 
  such that any real solution $\la$ of \eqref{equationGG} has unbounded variation on $I$; 
  see \cite[Example 4.4]{GhisiGobbino13}.
\end{itemize}  
\end{remark}

By a standard argument based on Fubini's theorem, also the following result for 
several variables was obtained in \cite{GhisiGobbino13}  (compare with \Cref{optimal2}).

\begin{theorem}
  \label{GhisiGobbino2}
  Let $k \in \N_+$, $\al \in (0,1]$, and $r=k+\al$.
  Let $f : \Om \to \R$ be a $C^{k,\al}$-function defined on an open set $\Om \subseteq \R^m$.
  Let $\la : \Om \to \R$ be any continuous function satisfying \eqref{equationGG}.  
  Then, for every relatively compact subset $V \Subset \Om$, we have $\nabla \la \in L^p_w(V,\R^m)$, 
  where $1/p + 1/r =1$, and  
  \begin{align*}
    \|\nabla \la\|_{p,w,V} 
    &\le C(m,k,\Om,V)\, \max\big\{(\Hoeld_{\al,\Om}(f^{(k)}))^{1/r}, 
        \|\nabla f\|_{L^\infty(\Om)}^{1/r}\big\}.
  \end{align*}
\end{theorem}

\section{Generic integrability along level sets} \label{sec:analytic}

The results of this section show that 
\[
\int_{\sgn(f)^{-1}(y)} |f|^{1/r} \, d\cH^{m-1} < \infty \quad \text{ for $\cH^1$-a.e.\ } y \in \S^1
\]
if $f : \R^m \supseteq \Om \to \C$ is of class $C^{k,\al}$ and $k+ \al \ge r$.
This is the main ingredient needed in the proof of \Cref{BVradicals} in \Cref{proof}.  
It can be understood as a complement of Sard's theorem for the sign function.

We believe that the results of this section are of independent interest and thus we formulate them in greater generality for maps $f : \R^m \supseteq \Om \to \R^{\ell+1}$. 
Then 
\begin{equation*} 
  \sgn(f) : \Om \setminus f^{-1}(0) \to \S^\ell, \quad \sgn(f) := \frac{f}{|f|}.
\end{equation*}
We will investigate the level sets of the sign $\sgn(f)$ and the norm $|f|$.
The proofs are based on \Cref{GhisiGobbino2} and the coarea formula.

\subsection{The coarea formula}

We will use the following version of the coarea formula due to \cite[Theorem 1.1]{MalySwansonZiemer03},
see also \cite{Federer69}.

Recall that a function $\tilde f$ is a \emph{precise representative} of $f \in L^1_{\on{loc}}(\Om)$ if
\[
  \tilde f (x) = \lim_{r\downarrow 0} \fint_{B_r(x)} f(y) \,dy
\] 
at all points $x$ where this limit exists. In the following we say that \emph{$f \in L^1_{\on{loc}}(\Om,\R^\ell)$ 
is precisely represented} if each of the component functions of $f$ is a precise representative. 
(In fact $f^{-1}(y)$ will depend on the 
representative of $f$.)

\begin{theorem}[Coarea formula] \label{coarea}
  Suppose that $m\ge \ell \ge 1$.
  Let $\Om \subseteq \R^m$ be open and 
  let $f \in W^{1,p}_{\on{loc}}(\Om,\R^\ell)$ be precisely represented, where either $p>\ell$ or $p \ge \ell = 1$.  
  Then $f^{-1}(y)$ is countably $\cH^{m-\ell}$ rectifiable for almost all $y \in \R^\ell$, and 
  for all measurable $E \subseteq \Om$,  
  \[
    \int_{E}  |J_\ell f(x)| \,dx = \int_{\R^\ell}  \cH^{m-\ell}(E \cap f^{-1}(y))\, dy.
  \] 
\end{theorem} 

Recall that $|J_\ell f(x)|$ is the square root of the sum of squares of the determinants of the $\ell \times \ell$ minors 
of the Jacobian of $f$.

The following \emph{change of variables formula} is an easy consequence of the coarea formula.

\begin{corollary} \label{changeofvariables}
  If $f$ is as in \Cref{coarea} and $g : \Om \to [0,\infty]$ is measurable, 
  then
  \[
    \int_\Om g(x) |J_\ell f(x)| \, dx = \int_{\R^\ell} \int_{f^{-1}(y)} g \, d\cH^{m-\ell} \, dy.
  \] 
\end{corollary}

We will apply these results only to continuous functions $f$ 
which evidently are precisely represented.

\subsection{Extension from Lipschitz domains} \label{sec:prep2}

It will be sometimes helpful to assume that functions are defined on $\R^m$ instead of on 
open subsets $\Om$ and have compact support. If $\Om$ is a bounded Lipschitz domain, then 
this is possible thanks to Whitney's extension theorem.

Let $\Om \subseteq \R^m$ be a bounded Lipschitz domain and
let $f \in C^{k,\al}(\overline \Om)$.
By Whitney's extension theorem, 
  $f$ admits a $C^{k,\al}$-extension $\hat f$ to $\R^m$
  such that 
  \begin{equation} \label{eq:whitney2} 
     \|\hat f\|_{C^{k,\al}(\R^m)}\le C \,
     \|f\|_{C^{k,\al}(\overline \Om)},
  \end{equation}
  for some constant $C=C(m,k,\al,\Om)$
  (cf.\ \cite[Theorem 4, p.177]{Stein70} and \cite[Theorem 2.64]{BrudnyiBrudnyi12Vol1}). 
  Let $\Om_1 :=\bigcup_{x \in \Om} B_1(x)$ be the open $1$-neighborhood of $\Om$.
  By multiplying $\hat f$ with a suitable cut-off function we may assume that $\on{supp}(\hat f) \subseteq \Om_1$, and that 
  \begin{equation} \label{eq:whitney3}
     \|\hat f\|_{C^{k,\al}(\overline \Om_1)}\le C(m,k,\al,\Om) \,
     \|f\|_{C^{k,\al}(\overline \Om)}.
  \end{equation}
  Clearly, these observations generalize to vector-valued functions.

\subsection{Level sets of the sign}
  
\begin{theorem} \label{prop:level}
Let $k \in \N_+$, $\al \in (0,1]$, and set $s = k +\al$. 
Let $\Om\subseteq \R^m$ be a bounded Lipschitz domain
and $f \in C^{k,\al}(\overline \Om, \R^{\ell+1})$, where $m \ge \ell \ge 1$. 
  Then there is a constant $C= C(m, \ell,k,\al,\Om)$ such that
  for each small $\ep >0$ 
  \begin{equation} \label{integralonlevel}
    \cH^\ell\Big(\Big\{y \in \S^\ell : \int_{\sgn(f)^{-1}(y)}  |f|^{\ell/s} \, d \cH^{m-\ell} 
    \ge  \ep^{-1} \, C \, \|f\|^{\ell/s}_{C^{k,\al}(\overline \Om)} \Big\}\Big) \le \ep.
  \end{equation}
\end{theorem}

\begin{proof} 
Let $f = (f_1,\ldots,f_{\ell+1}) \in C^{k,\al}(\overline \Om,\R^{\ell+1})$. Without loss of generality we may assume 
$f \not \equiv 0$. 
For convenience set $g:= \sgn(f) = f/|f| : \Om \setminus f^{-1}(0) \to \S^{\ell}$. 
Then, for all $j=1,\ldots, \ell+1$ and all $i = 1,\ldots,m$,
\begin{align*}
  \p_i g_j &= \frac{\p_i f_j}{|f|} - \frac{f_j}{|f|^3} \sum_{k=1}^{\ell+1} f_k \p_i f_k
  \quad  \implies \quad  
  |\p_i g_j| \le 2(\ell+1) \max_{1 \le k \le \ell+1} \frac{|\p_i f_k|}{|f|},  
\end{align*}
and 
\[
\p_i \big(|f_j|^{1/s}\big)  = \frac{1}{s}  \frac{f_j \p_i f_j}{|f_j|^{2-1/s}}.
\]
Let $h = (g_1,\ldots,g_\ell)$ consist of the first $\ell$ components of $g$. 
Then 
\begin{align*}
  |J_\ell h| \le C(m,\ell) \Big(\max_{\substack{1 \le j \le \ell\\ 1 \le i \le m}} |\p_i g_j|\Big)^\ell 
  \le C(m,\ell) \, \Big(\max_{\substack{1 \le k \le \ell+1\\ 1 \le i \le m}} \frac{|\p_i f_k|}{|f|} \Big)^\ell
\end{align*}
and consequently
\begin{align} \label{computation}
  |f|^{\ell/s} |J_\ell h| 
  &\le C(m,\ell)\, \Big(\max_{\substack{1 \le k \le \ell+1\\ 1 \le i \le m}} \frac{|\p_i f_k|}{|f|^{1-1/s}} \Big)^\ell   
  \le  C(m,\ell)\,s\, \Big(\max_{\substack{1 \le k \le \ell+1\\ 1 \le i \le m}} \big|\p_i \big(|f_k|^{1/s}\big)\big| \Big)^\ell.
\end{align}
By \Cref{GhisiGobbino2} (applied to an extension of $f$ as in \Cref{sec:prep2}) and by \eqref{computation}, 
we may conclude that  
\begin{equation} \label{eq:est1}
  \int_{\Om}  |f(x)|^{\ell/s} |J_\ell h(x)|\, dx \le C(m,\ell,k,\al,\Om)\,\|f\|^{\ell/s}_{C^{k,\al}(\overline \Om)}.
\end{equation}  
By the coarea formula (\Cref{changeofvariables}),  
\begin{equation} \label{eq:est2}
  \int_{\Om \setminus f^{-1}(0)}  |f|^{\ell/s} |J_\ell h(x)|\, dx  
  = \int_{\R^\ell} \int_{h^{-1}(y)}  |f|^{\ell/s} \, d \cH^{m-\ell} \, dy. 
\end{equation} 
Then \eqref{eq:est1} and \eqref{eq:est2} entail  
\begin{equation} \label{eq:hf}
  \int_{[-1,1]^\ell} \int_{h^{-1}(y)} |f|^{\ell/s} \, d \cH^{m-\ell} \, dy \le C(m,\ell,k,\al,\Om)\,\|f\|^{\ell/s}_{C^{k,\al}(\overline \Om)}.
\end{equation}
It follows that,
for all small enough $\ep>0$, 
\[
  \Big|\Big\{y \in [-1,1]^\ell : \int_{h^{-1}(y)} |f|^{\ell/s} \, d \cH^{m-\ell} \ge \ep^{-1} C(m,\ell,k,\al,\Om)\,\|f\|^{\ell/s}_{C^{k,\al}(\overline \Om)}  \Big\}\Big| \le \ep,
\]
which entails the statement of the lemma, since $g^{-1}(z_1,\ldots,z_{\ell+1}) \subseteq h^{-1}(z_1,\ldots,z_\ell)$ for all  
$z=(z_1,\ldots,z_{\ell+1}) \in \S^{\ell}$.   
\end{proof}

\begin{corollary} \label{cor:level}
  In the setting of \Cref{prop:level}, 
  \begin{equation} \label{integralonlevelcor}
  \int_{\sgn(f)^{-1}(y)}  |f|^{\ell/s} \, d \cH^{m-\ell} < \infty, \quad \text{ for $\cH^{\ell}$-a.e. } y \in \S^{\ell}, 
  \end{equation}
  and for every relatively compact open $K \Subset \Om \setminus f^{-1}(0)$, 
  \begin{equation} \label{eq:levelarg}
    \cH^{m-\ell}\big(K \cap \sgn(f)^{-1}(y)\big) < \infty, \quad \text{ for $\cH^{\ell}$-a.e. } y \in \S^{\ell}.  
  \end{equation}
\end{corollary}

\begin{proof}
  It is clear that \eqref{integralonlevel} implies \eqref{integralonlevelcor}.
  Let $K \Subset \Om \setminus f^{-1}(0)$ be open and relatively compact.
  By the coarea formula, where $h$ is the map defined in the proof of \Cref{prop:level},
  \begin{equation} \label{eq:est4}
    \int_{\R^\ell} \cH^{m-\ell}(K \cap h^{-1}(y))    \, dy= \int_{K} |J_\ell h(x)|\, dx 
  \end{equation}
  which is finite, since $|f|\ge \de > 0$ on $K$.
  So there is a subset $A = A_{f,K} \subseteq \R^\ell$ with $|\R^\ell \setminus A|=0$ such that 
  $\cH^{m-\ell}(K \cap h^{-1}(y)) < \infty$ for all $y \in A$. 
  This entails \eqref{eq:levelarg}.
\end{proof}

\subsection{Level sets of the norm}

The result of this section  
will not be needed in this paper but we think it is interesting in its own right.

\begin{theorem} \label{prop:levelgrowth}
  Let $k \in \N_+$, $\al \in (0,1]$, and set $s = k +\al$. 
  Let $\Om\subseteq \R^m$ be a bounded Lipschitz domain and
  $f \in C^{k,\al}(\overline \Om,\R^{\ell+1})$,  $f \not \equiv 0$.  
  Then there is a constant $C = C(m,\ell,k,\al,\Om)$ such that 
  for all $0<\ep \le 1$ and all small $\de >   0$ we have
  \begin{equation*} 
     \big|\big\{y \in (0,\de) : y^{1/s}\, \cH^{m-1}(|f|^{-1}(y)) 
     \ge \ep^{-1}C \,\|f\|^{1/s}_{C^{k,\al}(\overline \Om)} \big\}\big| \le \ep \de.
  \end{equation*}
\end{theorem}

\begin{proof}
  For $\de>0$ consider 
  \[
    I_{f}(\de) := \int_{0}^\de y^{1/s}\, \cH^{m-1}(|f|^{-1}(y)) \, dy.
  \]
  (Note that $|f|$ is Lipschitz and has a Lipschitz extension to $\R^m$, 
  and thus $y \mapsto y^{1/s}\, \cH^{m-1}(|f|^{-1}(y))$ is $\cL^1$-measurable.)
  Then 
  \begin{align*} 
  I_{f}(\de) 
  &= \int_{0}^\de \int_{|f|^{-1}(y)}  |f|^{1/s} \, d \cH^{m-1} \, dy
  \\
  &= \int_\R \int_{|f|^{-1}(y)} |f|^{1/s} \ind_{|f|^{-1}((0,\de))} \,   d \cH^{m-1} \, dy
  \\
  &=
  \int_{|f|^{-1}((0,\de))} |f(x)|^{1/s}\big|\nabla |f|(x)\big|\, dx, 
  \end{align*}
  where the last identity holds by the 
  coarea formula (\Cref{changeofvariables}).  

  On the set $\{x : f(x) \ne 0\}$,
  \begin{align*}
    \big|\p_i |f| \big| = \Big|\frac{\langle f , \p_i f \rangle }{|f|}\Big| \le |\p_i f|. 
  \end{align*}
  Hence, 
  \begin{align*}
       |f|^{1/s} \big|\nabla |f| \big| &\le \sqrt m \max_{1 \le i \le m} |f|^{1/s} \big|\p_i |f| \big|
       \le \sqrt m \max_{1 \le i \le m} |f| \frac{|\p_i f|}{|f|^{1-1/s}}
       \\
       &\le \sqrt {m(\ell+1)} |f| \max_{\substack{1 \le i \le m\\ 1 \le j \le \ell+1}}  \frac{|\p_i f_j|}{|f|^{1-1/s}}
       \\  
       &  
       \le s\, \sqrt {m(\ell+1)} |f| \max_{\substack{1 \le i \le m\\ 1 \le j \le \ell+1}}   
       \big| \p_i \big(|f_j|^{1/s}\big) \big| .
  \end{align*}
  By \Cref{GhisiGobbino2} (applied to an extension of $f$ as in \Cref{sec:prep2}), 
  \[
    \int_{\Om} \big| \p_i \big(|f_j|^{1/s}\big) \big|  \, dx 
    \le C(m,k,\al,\Om) \,\|f\|^{1/s}_{C^{k,\al}(\overline \Om)}.
  \]
  Consequently, we have 
  \begin{align} \label{eq:Ifupper}
  I_{f}(\de) 
  &\le  C \,\|f\|^{1/s}_{C^{k,\al}(\overline \Om)} \cdot \de,  
  \end{align} 
  for a constant $C=C(m,\ell,k,\al,\Om)$.

  Set $A_{\ep,\de} := \{y \in (0,\de) : y^{1/s}\, \cH^{m-1}(|f|^{-1}(y)) 
     \ge \ep^{-1}C \,\|f\|^{1/s}_{C^{k,\al}(\overline \Om)} \}$.
  For small $\de >0$, we have the lower bound
  \[
    I_f(\de) \ge \ep^{-1}C \,\|f\|^{1/s}_{C^{k,\al}(\overline \Om)} |A_{\ep,\de}|
  \] 
  which implies the assertion in view of \eqref{eq:Ifupper}. 
\end{proof}

\begin{corollary} \label{cor:levelgrowth}
  In the setting of \Cref{prop:levelgrowth} 
  let $A \subseteq [0,\infty)$ be such that $|A \cap [0,\ep)| = \ep$ for some $\ep>0$.
  Then there is a sequence $A \ni y_j \to 0$ with
  \begin{equation*} 
     \sup_j \Big(y_j^{1/s}\, \cH^{m-1}(|f|^{-1}(y_j)) \Big)\le C(m,\ell,k,\al,\Om) \,\|f\|^{1/s}_{C^{k,\al}(\overline \Om)}.
  \end{equation*}
\end{corollary}

\section{Background on functions of bounded variation} \label{preliminaries}

In this section we recall some facts on functions of bounded variation and fix notation. 
We follow the presentation in \cite{AFP00}.
In \Cref{sec:extensiontozerosets,sec:sufficientforBV} we prove some simple statements we shall need later on. 
They are probably well-known, but  
we include proofs, since we could not find them in the literature.

\subsection{Functions of bounded variation} \label{boundedvariation}  

Let $\Om \subseteq \R^m$ be open. 
A real-valued function $f \in L^1(\Om)$ is a \emph{function of bounded variation in $\Om$} if the distributional derivative of $f$ 
is representable by a finite Radon measure in $\Om$, i.e.,
\begin{equation*}
  \int_\Om  f \p_i \vh \,dx = - \int_\Om \vh\, d D_i f, \quad \text{ for all } \vh \in C^\infty_c(\Om), ~ i =1, \ldots,m, 
\end{equation*}
for some $\R^m$-valued measure $Df = (D_1 f,\ldots,D_mf)$ in $\Om$. 
The functions of bounded variation in $\Om$ form a vector 
space denoted by $BV(\Om)$.
The Sobolev space $W^{1,1}(\Om)$ is strictly contained in $BV(\Om)$; for $f \in W^{1,1}(\Om)$, $Df = \nabla f\, \cL^m$.

A real-valued function $f \in L^1_{\loc}(\Om)$ belongs to $BV_{\loc}(\Om)$ if $f \in BV(\Om')$ for every relatively compact 
$\Om' \Subset \Om$. 
We define $BV(\Om,\R^\ell) := BV(\Om,\R)^\ell$ and $BV(\Om,\C) := BV(\Om,\R^2)$. 

An element $f=(f_1,\ldots,f_\ell) \in L^1(\Om,\R^\ell)$ belongs to $BV(\Om,\R^\ell)$ if and only if the \emph{variation}  
\begin{equation*}
  \Var(f,\Om) := \sup \Big\{ \sum_{j=1}^\ell \int_\Om f_j \div \vh_j \, dx : 
      \vh \in C^\infty_c(\Om,\R^{m})^\ell,~ \|\vh\|_\infty \le 1 \Big\}
\end{equation*}
is finite. Then $\Var(f,\Om)$ coincides with the total variation measure $|Df|(\Om)$. 
It is lower semicontinuous in the $L^1_{\loc}(\Om,\R^\ell)$ topology, i.e., 
\begin{equation*}
  \Var(f,\Om) \le \liminf_{k \to \infty} \Var(f_k,\Om), \quad \text{ for } f_k \to f \text{ in } L^1_{\loc}(\Om,\R^\ell).
\end{equation*}
The space $BV(\Om,\R^\ell)$ endowed with the norm $\|f\|_{BV} := \|f\|_{L^1} + |Df|(\Om)$ is a Banach space.

\subsection{Approximate continuity and differentiability of functions of bounded variation} 

We say that $f \in L^1_{\loc}(\Om,\R^\ell)$ has an \emph{approximate limit at $x \in \Om$} if there is $z \in \R^\ell$ such that 
\begin{equation*}
   \lim_{r \downarrow 0} \fint_{B_r(x)} | f(y) - z | \,dy = 0.
 \end{equation*} 
The \emph{approximate discontinuity set $S_f$} is the set of $x \in \Om$, where this property does not hold.  
For $x \in \Om \setminus S_f$ the uniquely determined approximate limit $z$ is denoted by $\tilde f(x)$. 
The function $f$ is said to be \emph{approximately continuous at $x$} if $x \not\in S_f$ and $f(x) = \tilde f(x)$ (i.e., 
$x$ is a Lebesgue point of $f$). 
The set $S_f$ is an $\cL^m$-negligible Borel set and $\tilde f : \Om \setminus S_f \to \R^\ell$ is a Borel function 
which coincides $\cL^m$-a.e.\ in $\Om \setminus S_f$ with $f$.

We say that $x \in \Om$ is an \emph{approximate jump point of $f$} if there exist $a^\pm \in \R^\ell$ and $\nu \in \S^{m-1}$ 
such that $a^+ \ne  a^-$ and 
\begin{equation*}
  \lim_{r \downarrow 0} \fint_{B_r^\pm(x,\nu)} | f(y) - a^\pm | \,dy = 0,
\end{equation*}
where $B_r^\pm(x,\nu) := \{y \in B_r(x) :\pm \<y-x,\nu\> >0\}$. The triplet $(a^+,a^-,\nu)$ is denoted by 
$(f^+(x),f^-(x),\nu_f(x))$. 
The set of approximate jump points, denoted by $J_f$, 
is a Borel subset of $S_f$, the functions $f^\pm : J_f \to \R^\ell$ and $\nu_f : J_f \to \S^{m-1}$ are Borel functions.

Let $x \in \Om \setminus S_f$. Then $f$ is \emph{approximately differentiable at $x$} if there exists an $\ell \times m$ matrix 
$T$ such that 
\begin{equation*}
   \lim_{r \downarrow 0} \fint_{B_r(x)} \frac{| f(y) - \tilde f(x) - T(y-x) |}{r} \,dy = 0.
\end{equation*}
The matrix $T$ is uniquely determined. It is called the \emph{approximate differential of $f$ at $x$} 
and denoted by $\nabla f(x)$. 
The set of approximate differentiability points is denoted by $D_f$. It is a Borel set and $\nabla f : D_f \to \R^{\ell m}$ 
is a Borel function. 

By the Federer--Vol'pert theorem (cf.\ \cite[Theorem 3.78]{AFP00}),
for every $f \in BV(\Om,\R^\ell)$ the set $S_f$ is countably $\cH^{m-1}$-rectifiable, $\cH^{m-1}(S_f \setminus J_f) = 0$, and 
\[
  Df \restr J_f = \big((f^+ - f^-) \otimes \nu_f\big) \, \cH^{m-1} \restr J_f.
\]
By the Calder\'on--Zygmund theorem (cf.\ \cite[Theorem 3.83]{AFP00}), each
$f\in BV(\Om,\R^\ell)$ is approximately differentiable at $\cL^m$-a.e.\ point of $\Om$, and the approximate differential 
$\nabla f$ is the density of the absolutely continuous part of $Df$ with respect to $\cL^m$.

\subsection{Decomposition of $Df$} 

Let $f \in BV(\Om,\R^\ell)$. The Lebesgue decomposition provides a decomposition
\begin{equation*}
  Df = D^a f + D^s f,
\end{equation*}
where $D^a f$ is the absolutely continuous and $D^s f$ is the singular part of $Df$ with respect to $\cL^m$.
By defining 
\begin{equation*}
  D^j f := D^s f \restr J_f, \quad D^c f := D^s \restr (\Om \setminus S_f) 
\end{equation*}
we obtain the decomposition 
\begin{equation*}
  Df = D^a f + D^j f + D^c f,
\end{equation*}
noting that $Df$ vanishes on the $\cH^{m-1}$-negligible set $S_f \setminus J_f$ (cf.\ \cite[Lemma 3.76]{AFP00}). 
Then $D^j f$ and $D^c f$ are called the 
\emph{jump} and the \emph{Cantor part} of $Df$, respectively. We have
\begin{align}
\begin{split} \label{acjump}
  D^a f &= \nabla f \, \cL^m, \\
    D^j f &= \big((f^+ - f^-) \otimes \nu_f\big)\, \cH^{m-1} \restr J_f.
\end{split}
\end{align}
The Cantor part vanishes on sets which are $\si$-finite with respect to $\cH^{m-1}$ and 
on sets of the form $\tilde f^{-1}(E)$, where $E \subseteq \R^\ell$ with $H^1(E) = 0$  
(cf.\ \cite[Proposition 3.92]{AFP00}).

\subsection{Special functions of bounded variation} 

A function $f \in BV(\Om,\R^\ell)$ is said to be a \emph{special function of bounded variation} if $D^cf =0$; 
in this case we write 
$f \in SBV(\Om,\R^\ell)$. 
Then $SBV(\Om,\R^\ell)$ forms a closed subspace of $BV(\Om,\R^\ell)$.  
We have strict inclusions $W^{1,1}(\Om,\R^\ell) \subsetneq SBV(\Om,\R^\ell) \subsetneq BV(\Om,\R^\ell)$, in fact:

\begin{proposition}[{\cite[Proposition 4.4]{AFP00}}]
  Let $\Om \subseteq \R^m$ be open and bounded, and let $K \subseteq \R^m$ be closed with $\cH^{m-1}(K \cap \Om) < \infty$.
  Then any function $f : \Om \to \R$ which belongs to $L^\infty(\Om \setminus K) \cap W^{1,1}(\Om \setminus K)$ belongs also 
  to $SBV(\Om)$ and satisfies $\cH^{m-1}(S_f \setminus K) = 0$.
\end{proposition}

It is not hard to conclude from this proposition that the solutions of $Z^n = x$, $x \in \C$, admit representations in 
$SBV_{\loc}$; but see \Cref{example}.

\subsection{The chain rule} 
 
Let $\Om \subseteq \R^m$ be a bounded open set. 
It is not hard to see that the composite $h = f \o g$ of a function $g \in BV(\Om, \R^\ell)$ and a Lipschitz function 
$f : \R^\ell \to \R^k$ belongs to $BV(\Om,\R^k)$ and that $|Dh| \le \Lip(f) |Dg|$. We shall need a more precise \emph{chain rule}. 
For our purpose it is enough to assume that $f$ is $C^1$; for the general case see \cite[Theorem 3.101]{AFP00}.

It is convenient to distinguish between the \emph{diffuse} part $\tD g := D^a g + D^c g$ and the jump part $D^j g$ of 
$Dg$, since they behave differently.

\begin{theorem}[{\cite[Theorem 3.96]{AFP00}}] \label{thm:chainrule}
  Let $g \in BV(\Om, \R^\ell)$ and let $f \in C^1(\R^\ell,\R^k)$ be a Lipschitz function satisfying $f(0)=0$ 
  if $|\Om| = \infty$. 
  Then $h = f \o g$ belongs to $BV(\Om,\R^k)$ and 
  \begin{align} \label{chainrule}
    \begin{split}
      \tD h &= \nabla f(g) \nabla g \, \cL^m + \nabla f (\tilde g)\, D^c g = \nabla f  (\tilde g)\,\tD g,\\
      D^j h &= (f(g^+) - f(g^-)) \otimes \nu_g \, \cH^{m-1} \restr J_g.    
    \end{split}  
  \end{align}   
\end{theorem}   

The following \emph{product rule} is an immediate consequence.
For $g_1,g_2 \in BV(\Om)$, $g = (g_1,g_2)$, and $f(y_1, y_2) = y_1y_2$, the product $g_1 g_2$ belongs to $BV(\Om)$ and 
\begin{align} \label{productrule}
  \begin{split}
    \tD (g_1 g_2) &= \tilde g_1 \tD g_2 + \tilde g_2 \tD g_1, \\
    D^j (g_1 g_2) &= (f(g^+) - f(g^-))  \nu_g \, \cH^{m-1} \restr J_g.
  \end{split}
\end{align}

If $g \in W^{1,p}(\Om,\R^\ell)$, $f : \R^\ell \to \R^k$ is Lipschitz, and $f \o g \in L^{p}(\Om,\R^\ell)$, 
then $f \o g \in W^{1,p}(\Om,\R^\ell)$ and the chain rule reduces to 
\begin{equation} \label{chainruleW1p}
   	\nabla (f \o g)(x) = \nabla f(g(x)) \cdot \nabla g(x) \quad \text{ for a.e. } x \in \Om;
\end{equation}   
e.g.\ \cite[Theorem 2.1.11]{Ziemer89}.

\subsection{Extension to zero sets} \label{sec:extensiontozerosets}

The \emph{pointwise variation} of a function $f : I \to \R^\ell$ on an open interval $I = (a,b)$, for $a<b \in \R_{\pm \infty}$ 
is defined by 
\[
  \on{pVar}(f,I) := \sup \Big\{\sum_{i=1}^{n-1} | f(t_{i+1}) - f(t_i)| : n \ge 2,\, a<t_1<\cdots < t_n <b\Big\}.
\] 
For open $\Om \subseteq \R$ one sets 
\[
  \on{pVar}(f,\Om) :=  \sum_I \on{pVar}(f,I), 
\]
where $I$ runs through all connected components of $\Om$. The \emph{essential variation} 
\[
  \on{eVar}(f,\Om) := \inf \big\{ \on{pVar}(g,\Om)  : g = f \text{ $\cL^1$-a.e.\ in } \Om\big\}
\]
coincides with the variation $\on{Var}(f,\Om)$ if $f$ is in $L^1_{\loc}(\Om)$; see \cite[Theorem 3.27]{AFP00}.

\begin{lemma} \label{lem:BVzeroext}
  Let $\Om_0 \subseteq \Om \subseteq \R$ be bounded open subsets. Let $f : \Om \to \C$ be a function in $L^1_{\loc}(\Om)$ 
  such that $f|_{\Om_0}$ has bounded (resp.\ pointwise) variation and $f$ vanishes and is continuous at all points of 
  $\Om \setminus \Om_0$. Then $f$ has bounded (resp.\ pointwise) variation on $\Om$ and 
  $\on{Var}(f,\Om_0)=\on{Var}(f,\Om)$ (resp.\ $\on{pVar}(f,\Om_0)=\on{pVar}(f,\Om)$).
\end{lemma}

\begin{proof}
  We prove first the statement about the pointwise variation. 
  Let $I=(a,b)$ be a connected component of $\Om$ and let $\{J_k\}_k$ be the collection of all components 
  of $\Om_0$ contained in $I$. 
  Suppose $a<t_1<\cdots < t_n <b$ and let $\ep>0$.
  We have a subdivision into a finite number $m$ of maximal chains $t_{i} < t_{i+1}< \cdots <t_{i+j}$ contained in some $J_k$ 
  and the remaining $t_p \in \Om \setminus \Om_0$. 
  If $t_{i} < t_{i+1}< \cdots <t_{i+j}$ is a chain contained in $J_k = (a_k,b_k)$, there exist 
  $t_i^-$ and $t_{i+j}^+$ such that $a_k < t_i^- <t_{i} <\cdots <t_{i+j} < t_{i+j}^+ < b_k$ and  
  $|f(t_i^-)|\le \ep/(2m)$ and $|f(t_{i+j}^+)|\le \ep/(2m)$. 
  Thus,
  \begin{align*}
    \sum_{i=1}^{n-1} | f(t_{i+1}) - f(t_i)| \le \sum_k \on{pVar}(f,J_k) + \ep
  \end{align*}
  and consequently $\on{pVar}(f,I) \le \sum_k \on{pVar}(f,J_k)$. This implies $\on{pVar}(f,\Om_0)=\on{pVar}(f,\Om)$.

  Suppose that $\Var(f,\Om_0)<\infty$. Then  
  \begin{align*}
    \Var(f,\Om_0) &= \inf \{\on{pVar}(g,\Om_0) : g = f ~\cL^1\text{-a.e.\ in } \Om_0\}
  \end{align*}
  and there exists $g$ coinciding $\cL^1$-a.e.\ with $f$ in $\Om_0$ such that $\on{pVar}(g,\Om_0)<\infty$. 
  If we extend such $g$ by $0$ on $\Om\setminus \Om_0$, then 
  $\on{pVar}(g,\Om_0)=\on{pVar}(g,\Om)$ by the first part. 
  (For the proof of the first part it is enough that, if $t_0 \in \Om \setminus \Om_0$, 
  then for all $\ep >0$ there is $\de >0$ such that $|f(t)| \le \ep$ for $\cL^1$-a.e.\ $t$ with $|t-t_0| \le \de$.)
  This implies $\on{Var}(f,\Om_0)\ge \on{Var}(f,\Om)$. The opposite inequality is trivial.
  The proof is complete.
\end{proof}

The lemma implies a similar result for functions in several variables. 
If $\Om \subseteq \R^m$ is an open set and $v \in \S^{m-1}$, we denote by $\Om_v$ the orthogonal projection of $\Om$ 
onto the hyperplane orthogonal to $v$. For each $y \in \Om_v$ we have the section $\Om_y^v := \{t \in \R : y+tv \in \Om\}$.
If $f$ is a function defined on $\Om$, then $f_y^v := f(y+tv)$ is a function defined on $\Om_y^v$.  

\begin{proposition} \label{prop:extension}
  Let $\Om_0 \subseteq \Om \subseteq \R^m$ be bounded open subsets. Let $f : \Om \to \C$ be a function 
  such that $f|_{\Om_0}$ has bounded variation and $f$ vanishes and is continuous at all points of 
  $\Om \setminus \Om_0$. Then $f$ has bounded variation on $\Om$ and 
  $\on{Var}(f,\Om_0)=\on{Var}(f,\Om)$.
\end{proposition}

\begin{proof} 
  Since $f|_{\Om_0} \in BV(\Om_0)$, \cite[Remark 3.104]{AFP00} implies that there exist $m$ linearly independent vectors $v_i$ such that 
  $f_y^{v_i} \in BV(\Om_{0,y}^{v_i})$ for $\cL^{m-1}$-a.e.\ $y \in \Om_{0,v_i}$ and 
  \[
    \int_{\Om_{0,v_i}} |D f_y^{v_i}|(\Om_{0,y}^{v_i}) \, dy < \infty, \quad \text{ for } i = 1,\ldots,m. 
  \]
  By \Cref{lem:BVzeroext}, each such $f_y^{v_i}$ extends by $0$ to a function $f_y^{v_i} \in BV(\Om_{y}^{v_i})$ 
  with $|D f_y^{v_i}|(\Om_{0,y}^{v_i}) = |D f_y^{v_i}|(\Om_{y}^{v_i})$. 
  Since $f_y^{v_i} \equiv 0$ for all $y \in \Om_{v_i} \setminus \Om_{0,v_i}$, we have
  \[
    \int_{\Om_{v_i}} |D f_y^{v_i}|(\Om_{y}^{v_i}) \, dy = \int_{\Om_{0,v_i}} |D f_y^{v_i}|(\Om_{0,y}^{v_i}) \, dy < \infty, \quad \text{ for } i = 1,\ldots,m. 
  \]  
  Thus $f \in BV(\Om)$, again
  by \cite[Remark 3.104]{AFP00}. The identity $\on{Var}(f,\Om_0)=\on{Var}(f,\Om)$ follows from \cite[Theorem 3.103]{AFP00}.  
\end{proof}

\subsection{A sufficient condition for bounded variation} \label{sec:sufficientforBV}

Let us recall a version of the Gauss--Green theorem which is a special case of 
the result given in \cite[p.~314]{Federer52}. 

\begin{theorem}[Gauss--Green theorem] \label{GaussGreen}
  Let $\Om \subseteq \R^m$ be a bounded open set
  with $\cH^{m-1}(\p \Om) < \infty$. Assume that there is a closed set $E \subseteq \R^m$ such that 
  $\p \Om \setminus E$ is a $C^1$-hypersurface and $\cH^{m-1}(E) = 0$.
  Then for each $f \in C^1(\overline \Om,\R^m)$,
  \begin{equation*}
    \int_\Om \div f\, dx =  - \int_{\p \Om} \<f , \nu_\Om\>\, d\cH^{m-1},
  \end{equation*}
  where $\nu_\Om$ is the inner unit normal to $\Om$. 
\end{theorem}

The following consequence will be used several times.

\begin{proposition} \label{lem:key}
  Let $\Om_0 \subseteq \R^m$ be an open set and let $E$ be a closed  
  $C^1$-hypersurface in $\Om_0$. 
  Let $f=(f_1,\ldots,f_\ell) \in L^1(\Om_0,\R^\ell)$ be such that $f \in W^{1,1}(\Om_0 \setminus E)$, 
  $f$ is $C^1$ on $\Om_0 \setminus E$ and extends together with its partial derivatives continuously to $E$ from 
  both sides, and 
  \[
    \int_E |f| \, d \cH^{m-1} < \infty.
  \] 
  Then $f \in BV(\Om_0,\R^\ell)$ and 
  \[
    \Var(f,\Om_0) \le { C(m,\ell)}\, \Big( \int_{\Om_0 \setminus E} |\nabla f| \,dx +  2 \int_E |f| \, d \cH^{m-1} \Big).
  \]  
\end{proposition}

\begin{proof}
  Let $\vh \in C^\infty_c(\Om_0,\R^m)^\ell$ with $\|\vh\|_\infty \le 1$.
  There exists a open relatively compact subset $\Om_\vh \Subset \Om_0$ which contains $\on{supp} \vh$ 
  and has smooth boundary such that
  $\cH^{m-1}(\p \Om_\vh \cap E) =0$. 
  Let $\cC$ be the set of connected components of the open set $\Om_\vh \setminus E$.
  Then, by the Gauss--Green theorem \ref{GaussGreen}, for each $j = 1,\ldots, \ell$,
  \begin{align*}
     \int_{\Om_0} f_j \div \vh_j \, dx &=  \int_{\Om_\vh} f_j \div \vh_j \, dx = \sum_{C \in \cC} \int_{C} f_j \div \vh_j \, dx
    \\
    &= \sum_{C \in \cC}\Big(-   \int_{C} \langle \nabla f_j , \vh \rangle \, dx 
        - \int_{\p C}  f_j \langle  \vh , \nu \rangle \, d\cH^{m-1} \Big)
    \\
    &= -   \int_{\Om_\vh} \langle \nabla f_j , \vh \rangle \, dx 
        - \sum_{C \in \cC} \int_{\p C}  f_j \langle  \vh , \nu \rangle \, d\cH^{m-1}
    \\
    &= -   \int_{\Om_0} \langle \nabla f_j , \vh \rangle \, dx 
        - \sum_{C \in \cC} \int_{\p C \cap E}  f_j \langle  \vh , \nu \rangle \, d\cH^{m-1},          
  \end{align*}
  where $\nu$ is the inner unit normal to $C$.
  The statement follows.
\end{proof}

\section{Proof of \texorpdfstring{\Cref{BVradicals}}{Theorem 1.3}} \label{proof}

Let $r>0$ and $m\in \N_{\ge 2}$. 
  Let $k \in \N$ and $\al \in (0,1]$ be such that 
  \[
    k+\al \ge \max\{r,m\}. 
  \]
  Let $\Om \subseteq \R^m$ be a bounded Lipschitz domain and   
  let $f \in C^{k,\al}(\overline \Om)$.
  By \Cref{sec:prep2}, we may assume that 
  $f$ has compact support contained in $\Om$. 
Let 
\[
  r = \ell+\be 
\]
be the unique representation of $r$, where $\ell \in \N$ and $\be \in (0,1]$.  
Consider the equation
\begin{equation} \label{equation2}
  Z^r = f.
\end{equation}
In order to single out some trivial cases we make the following distinction.
\begin{description}
  \item[Case 1] $\ell=0$, $1/\be \in \N$. In this case $f^{1/\be}$ is a solution of \eqref{equation2} in the same differentiability 
  class as $f$.
  \item[Case 2] $\ell=0$, $1/\be \not \in \N$. Then $1/\be > 1$ and $f^{\lfloor 1/\be \rfloor}\cdot \la$ is a 
  solution of \eqref{equation2}, provided that $\la$ solves \eqref{equation2} for $r = 1/\{1/\be\}$. 
  The factor $f^{\lfloor 1/\be \rfloor}$ is in the differentiability class of $f$. 
  Since $\{1/\be\} \in (0,1)$, the existence and regularity of $\la$ is covered by the next case.
  \item[Case 3] $r>1$ or equivalently $\ell \in \N_+$, $\be \in (0,1]$.  
\end{description}
Henceforth we restrict to the Case 3.

  Let $\Om_0 := \Om \setminus f^{-1}(0)$,
  and consider the sign map $\sgn(f)$  defined in \eqref{sgnf}.
   By \Cref{discont}, for each regular value $y \in \S^1$ of $\sgn(f)$ the set $E := \sgn(f)^{-1}(y)$ 
   is a closed $C^k$-hypersurface of $\Om_0$ (possibly empty) and there is a continuous function 
  $\la : \Om \setminus E \to \C$ such that $\la ^r = f$; note that $\la$ is of class $C^{k}$ on $\Om_0 \setminus E$. 
  If we write $f = u + i v$, then, as in \eqref{computation},
  \begin{align*}
    |\nabla \la| = \frac{1}{r} \frac{|\nabla f|}{|f|^{1-1/r}}  \le \Big( \big| \nabla \big(|u|^{1/r}\big) \big| 
    + \big| \nabla \big(|v|^{1/r}\big) \big| \Big).
  \end{align*}
  By \Cref{GhisiGobbino2}, we may conclude that $\nabla \la \in L^p_w(\Om_0\setminus E,\C^m)$ with  
  \begin{equation} \label{eq:corGGweak}
    \|\nabla \la\|_{L^p_w(\Om_0\setminus E)} \le C(m,k,\al,\Om)\, \|f\|_{C^{k,\al}(\overline \Om)}^{1/r}
  \end{equation}
  where $p = r/(r-1)$.
  Here we use the fact that 
  $C^{k,\al}(\ol \Om)$ is continuously embedded in $C^{\ell,\be}(\ol \Om)$,  
  as $\Om$ is quasiconvex; cf.\ \cite[Proposition 3.7]{LlaveObaya99}.

  There exists $y \in \S^1$ which is a regular value of $\sgn(f)$ and such that $E = \sgn(f)^{-1}(y)$ 
  satisfies  
  \begin{gather}
   \int_{E}  |f|^{1/r} \, d \cH^{m-1} \le C(m,k,\al,\Om)\, \|f\|_{C^{k,\al}(\ol \Om)}^{1/r},  \label{level0}\\
   \label{eq:levelarg2}
   \cH^{m-1}(K \cap E) < \infty, \quad \text{ for every relatively compact } K \Subset \Om_0.  
  \end{gather}
  Here we apply \Cref{prop:level} and \Cref{cor:level},
  using $s := k+\al \ge r$ and the continuous inclusion of   
  $C^{k,\al}(\ol \Om)$ in $C^{\ell,\be}(\ol \Om)$.
  
  Then \Cref{lem:key} implies that $\la \in BV(\Om_0)$ and 
  \begin{align*}
    \Var(\la,\Om_0) &\le { C(m)} \, \Big(\int_{\Om_0 \setminus E} |\nabla \la| \,dx +  2 \int_E |f|^{1/r} \, d \cH^{m-1} \Big)
    \\ &\le  C(m,k,\al,\Om)\, \|f\|_{C^{k,\al}(\overline \Om)}^{1/r},
  \end{align*}
  by \eqref{eq:corGGweak} and \eqref{level0}. By \Cref{prop:extension},
  we may conclude that $\la \in BV(\Om)$ and we obtain the 
  uniform bound \eqref{eq:uniformbound}, since $\|\la\|_{L^1(\Om)} \le |\Om| \,\|f\|_{L^\infty(\Om)}^{1/r}$.

  By construction, $S_\la = J_\la = E$.
  To see that the Cantor part $D^c \la$ vanishes consider the disjoint union
  $\Om = (\Om_0 \setminus E) \cup E \cup  f^{-1}(0)$.
  Now $\la$ is of class $W^{1,1}$ in $\Om_0 \setminus E$, $E$ is $\si$-finite with respect to $\cH^{m-1}$, and 
  $\la$ is continuous on the set $f^{-1}(0) = \la^{-1}(0)$.
  Thus $D^c \la$ is zero on each of these sets, by \cite[Proposition 3.92]{AFP00}.
  The proof of \Cref{BVradicals} is complete. \qed

\section{Formulas for the roots} \label{sec:formulas}

\subsection{Idea of the proof of \Cref{thm:main}}
The main idea of the proof of \Cref{thm:main} lies in the reduction to the radical case using the formulas 
for the roots of the universal polynomial $P_a$, 
$a\in \C^n$, 
given in \cite[Theorem 1.6]{ParusinskiRainerAC}.  
These formulas express the roots of $P_a$  
 as finite sums of functions analytic in the radicals of local coordinates 
 on a resolution space (i.e.\ a blowing-up of $\C^n$), 
 see \cite{ParusinskiRainerAC} and \Cref{roots} below.  
Thus using the radical case we may choose a parameterization of every such 
locally defined summand, by choosing a parameterization of the radicals, 
but then a new difficulty arises. 
It comes from the fact that these summands can be  defined only locally on the resolution space. 
Therefore  we have to cut and paste them, introducing new  discontinuities and taking care of the 
integrability of the new jumps.  
This step will be done in \Cref{sec:cutpaste}.  

\begin{remark}
This local decomposition of the roots into finitely many summands is obtained 
in \cite{ParusinskiRainerAC} by a repeated ``splitting'' of $P_a$, 
a procedure  that reflects how the roots are regrouped in clusters. 
The summands of \Cref{varphik} represent the arithmetic means of such clusters.
\end{remark}

\subsection{Tower of smooth principalizations and formulas for the roots}

Let \begin{align}\label{polynomial}
P_a (Z)=Z^n+\sum _{j=1}^n a_j Z^{n-j}, 
\end{align}
where $a_j\in \C$.  
We defined in \cite {ParusinskiRainerAC}  the generalized discriminant ideals $\sD_{\kk} \subseteq \C[a]= \C[a_1, ..., a_n]$, 
$\kk = 2, ... , n$,  that, in particular, satisfy  the following properties.   For each $\kk$ the zero set 
$V(\sD_\kk)$ of $\sD_\kk$ is 
exactly the set of those $a$ for which $P_a$ has at most  $\kk -1$  distinct roots and therefore 
$V(\sD_{\kk-1}) \subseteq V(\sD_\kk)$.  The top ideal $\sD_{n} $ is  principal and  generated by a power of  
the discriminant of $P_a$.  The other discriminant ideals are not principal.  

For  the ideals $\sD_{\kk}$, $\kk =2, \ldots, n$, 
we constructed in  \cite {ParusinskiRainerAC} a tower of smooth principalizations  
\begin{align}\label{tower}
M_1 = \C^n  \stackrel{ \sigma_{2,1} }\longleftarrow M_2  \stackrel{ \sigma_{3,2} }\longleftarrow  M_3 
\stackrel{ \sigma_{4,3} }\longleftarrow  \cdots \stackrel{ \sigma_{n,n-1} }\longleftarrow M_{n} ,
\end{align}
where each $\sigma_{\kk,\kk-1}$ is the composition of blowing-ups with smooth centers and for every 
$\sigma_\kk = \sigma _{2,1} 
\circ \sigma _{3,2} \circ \cdots \circ \sigma_{\kk,\kk-1}$ the ideal  $\sigma_\kk^*(\sD_{\kk })$ is principal.  
We write $\sigma_{n,\kk} = \sigma_{\kk+1,\kk} \circ \cdots \circ \sigma_{n,n-1}$.   Such a tower of smooth 
principalizations exists by the classical results of resolution of singularities.  

By construction each $p\in M_\kk$ comes with a privileged system of local analytic coordinates on $M_\kk$ at $p$.  
It is obtained as follows (see \Cref{extendedchain}).  
The map $\sigma_{\kk}$  is itself a blowing-up of an ideal $\sK_{\kk}  \subseteq \C [a]$ such that 
$V(\sK_\kk )  = V(\sD_\kk )$.  We fix such an ideal $\sK_{\kk} $ as well as a system of its generators 
$\sK_\ell=(h_1, \ldots, h_k)$,  
where $h_i \in \C[a]$.  
The  pull-back $\sigma_\kk^*(\sK_{\kk })$ is principal, so it is generated in a neighborhood of $p$ by one of 
the $h_i\circ \sigma_\kk$.  
We denote this $h_i$ simply by $h$.  Also, in order to simplify the notation, we use the same symbol for a polynomial on $\C^n$ 
and its pull-back to $M_\kk$ (in particular, we will write $h$ instead of $h_i\circ \sigma_\kk$) if no confusion is possible.  

Similarly, $\sigma_\kk^*(\sD_{\kk })$ is principal and we fix a polynomial $f\in \C [a]$ that generates 
$\sigma_\kk^*(\sD_{\kk })$ 
at $p$.
Let $E = V(\sigma_\kk^*(\sD_{\kk }))$.  
There is a neighborhood $\cU$ of $p$ in $M_{\kk}$, and a coordinate system 
$y_1, \ldots, y_n$ on $\cU$, such that $y_i= P_i /h^s$,  $P_i \in \C[a]$, $s\in \N$, 
$E \cap \cU = \{y_1 \cdots y_r=0\}$ 
and such that on $\cU$ 
\begin{align}\label{normalcrossings}
 f = unit \cdot \prod_{i=1}^r y_i^{n_i},  \quad h = unit \cdot \prod_{i=1}^r y_i^{m_i}, 
\end{align}
with $n_i>0$ and $m_i> 0$.  Here by a unit we mean an analytic function defined and  nowhere vanishing  on $\mathcal U$. 

 \begin{remark}
The choice of $h$, $f$, and $P_i$ is not unique, but $f$ and $h$ are well defined at $p$  up to a local analytic unit.  
Similarly, $y_i$, for $i\le r$, are defined by their zero sets up to a unit.  
These zero sets, that is the components of the exceptional divisor of $\sigma_\kk$,  are well-defined. 
 \end{remark}

\begin{definition}
By a \emph{chain $\sfC= (p_{\kk}, f_{\kk}, h_\kk, P_{\kk,i}, s_\kk , r_\kk)$ for $p_n\in M_n$} we mean the 
points  $p_\kk : = \sigma _{n,\kk} (p_n) $, $\kk=1, \ldots, n$, and the local data 
$(f_{\kk},  h_\kk,  P_{\kk, i}, s_\kk, r_\kk)$ for $p_\kk$.  We complete this data for $\kk=1$ by putting 
 $f_1 =h_1=1$, $P_{1,i}=a_i$, and $s_1=r_1=0$.    
\end{definition}

We pull back the polynomial $P_a$ onto 
$M_\kk$ via $\sigma_\kk$, 
$$
P_{\sigma_\kk^*(a)} (Z) = Z^n + \sum _{i=1}^n( a_i\circ \sigma_\kk) Z^{n-i}.
$$
The  roots of $P_{\sigma_n^*(a)}$ are the pull-backs of the roots of $P_a$.

\begin{theorem}[{\cite[Theorem 1.6]{ParusinskiRainerAC}}]\label{roots}
 We may associate with every chain $\sfC= (p_{\kk}, f_{\kk}, h_\kk, P_{\kk,i}, s_\kk , r_\kk)$ 
convergent power series $\psi_\kk$, 
integers $q_\kk\ge 1$, and positive exponents $\al_\kk \in  \frac 1 {q_\kk}\N_{>0}$, such that the following holds.
The roots of  
 $P_{\sigma_n^*(a)} $ in a neighborhood of  $p_n$ are given by  
 \begin{align}\label{sums}
 \sum_{\kk=1}^n A_\kk \,  \, {\varphi_\kk  \circ \sigma_{n,\kk}} ,
\end{align} 
 where  $ A_\kk\in \Q$ and 
 \begin{align}\label{varphik}
\varphi_\kk & = f_\kk^ {\al_\kk}  \psi_\kk \big(y_{\kk,1}^{1/q_\kk}, \ldots , 
y_{\kk, r_\kk}^{1/q_\kk}, y_{\kk,r_\kk+1}, \ldots , y_{\kk,n}\big) . 
\end{align}
\end{theorem}

The meaning of the radicals in \eqref{varphik} is explained in  \cite[Remark 1.7]{ParusinskiRainerAC}.  
There are neighborhoods $\mathcal U_\kk$  of $p_\kk$, satisfying 
$\sigma _{\kk,\kk-1} (\mathcal U_{\kk})\subseteq \mathcal U_{\kk-1}$, and their branched covers 
$\tau_\kk : \widetilde {\mathcal U}_\kk \to \mathcal U_\kk$, given by the formulas 
\begin{equation}\label{branching}
y_{\kk,i} = \begin{cases}
 t_{i}^{q_\kk} & \text {  if } i\le r_\kk,  \\
 t_i  &  \text{ if }    i > r_\kk+1, 
\end{cases}
\end{equation}
such that  $\varphi_\kk$ can be interpreted as an analytic function on $\widetilde {\mathcal U}_\kk$.    
Since $\sigma_{\kk+1,\kk}^{-1}  (f_\kk^{-1}(0) )\subseteq f_{\kk+1} ^{-1} (0)$, the composite 
$y_{\kk,i}\circ  \sigma_{\kk+1,\kk}$, 
for $i\le r_\kk$,  is a normal crossings in $y_{\kk+1,1}, \ldots , y_{\kk+1, r_{\kk+1}}$ and therefore, we may 
suppose that $\sigma_{\kk+1,\kk} \circ \tau _{\kk+1}$ factors through $\tau_\kk$, changing $q_{\kk+1}$ if necessary.  
  Thus we obtain a sequence of branched covers $\tau_i$ making the following diagram commutative.
\[
  \xymatrix@C=1.5cm{
\widetilde   {\mathcal U}_{1} & \ar[l]_{\widetilde \si_{1,2}} \widetilde  {\mathcal U}_2 &
 \ar[l]_{\widetilde \si_{3,2}} \widetilde  {\mathcal U}_3 
&  \ar[l]_{\widetilde \si_{4,3}} \cdots & \ar[l]_{\widetilde \si_{n-1,n-2}} \widetilde {\mathcal U}_{n-1} 
&  \ar[l]_{\widetilde \si_{n,n-1}} \widetilde {\mathcal U}_n \\
  \ar@{=}[u]  {\mathcal U}_{1} & \ar@{<-}[u]_{\tau_{2}} \ar[l]_{\si_{1,2}}  {\mathcal U}_2 
  &  \ar@{<-}[u]_{\tau_3}  \ar[l]_{\si_{3,2}}  {\mathcal U}_3 &  \ar[l]_{\si_{4,3}} \cdots 
  & \ar@{<-}[u]_{\tau_{n-1}} \ar[l]_{\si_{n-1,n-2}}  {\mathcal U}_{n-1} & \ar@{<-}[u]_{\tau_n} \ar[l]_{\si_{n,n-1}} 
   {\mathcal U}_n  
  }
\]
Then Theorem  \ref{roots} says that the roots of  $P_{\tilde \si_n^*(a)} $,  where  
$\tilde \si_n = \si _n \circ \tau_n$,  are  combinations of analytic functions on 
$\widetilde {\mathcal U}_n$ that are pull-backs of such $\varphi_\kk$.

\begin{definition}\label{extendedchain}
By an \emph{extended chain} 
$\sfE= (p_{\kk}, f_{\kk}, h_\kk, P_{\kk,i}, s_\kk , r_\kk, \mathcal U_\kk)$ for $p_n\in M_n$ we mean a 
chain $\sfC= (p_{\kk}, f_{\kk}, h_\kk, P_{\kk,i}, s_\kk , r_\kk)$ and a system of neighborhoods  
  $\mathcal U_{\kk}$ of $p_\kk$ as above. The $y_{\kk,i}= P_{\kk,i} /h_\kk^{s_\kk}$, $i=1, \ldots, n$, 
are called \emph{a privileged system of coordinates} on  $\mathcal U_\kk$. 
\end{definition}

By \eqref{normalcrossings}, we may express $\varphi_\kk $ of \eqref{varphik} as follows 
\begin{align}\label{varphiksecond}
\varphi_\kk = h_\kk^ {\tilde \al_\kk} \tilde \psi_\kk \big(y_{\kk,1}^{1/\tilde q_\kk}, \ldots , y_{\kk,r_\kk}^{1/\tilde q_\kk}, 
y_{\kk,r_\kk+1}, \ldots , y_{\kk,n}\big), 
\end{align}
where $\tilde \al_k \in \frac 1 {\tilde q_\kk} \N_{>0}$ and $\tilde q_\kk$ is a positive integer possibly much 
 bigger than $q_\kk$.

\section{Proof of \texorpdfstring{\Cref{thm:main}}{Theorem 1.1}} \label{generalproof}

We assume that the coefficients $a_j$, $j =1,\ldots,n$, 
have compact support in $\R^m$ 
and that the image of $a$ is contained in the closed unit ball $\ol \B \subseteq \C^n$. 
The general case will be reduced to this case in \Cref{subs:GeneralCase}.  
We assume also that the $a_j$ are of class $C^{k-1,1}$ 
with $k=k(n,m)$ defined in \Cref {ssec:kp}.

In the following we shall be dealing with multi-valued functions arising from complex radicals, 
their composition with single-valued 
functions, and their addition and multiplication.

\subsection{In one blow-up chart} \label{sec:onechart}

We will use the notation of \Cref{sec:formulas}. 
Let $(\cU,\mathbf{y})$ be a chart on $M_\ell$ 
with a privileged system of coordinates $y_i = P_i/h^s$.
We may assume that $\mathbf {y}(\cU)$ coincides with 
the open unit ball $\B$ in $\C^n$. 

Let $a : \R^m \to \C^n$ be sufficiently differentiable with compact support. We distinguish between the chart map $\mathbf {y}$ and the composite 
map $\ul y = (\ul y_1,\ldots,\ul y_n)$ given by   
\[
\ul h := h\circ a, \quad \ul P_i := P_i \circ a, \quad 
\ul y_i := y_i \circ a=  \frac{\ul P_i}{\ul h^{s}}.
\]
The map $\ul y$ is defined and continuous on the set $\{x \in \R^m : \ul h(x) \ne 0\}$. 
Consider the open subset  
\begin{equation} \label{eq:defOm0}
  \Om_0 := \ul y^{-1}(\B) = \big\{x \in \R^m : |\ul y(x)| < 1,\, \ul h(x) \ne 0\big\} = a^{-1}(V_0),
\end{equation}
where
$V_0 = \{a \in \C^n : |y(a)| < 1,\, h(a) \ne 0\}$. 
The image of the multi-valued map 
\[
  \ul y^{1/\tilde q} =\big(\ul y_1^{1/\tilde q},\ldots,\ul y_r^{1/\tilde q}, 
  \ul y_{r+1},\ldots, \ul y_n\big) : \Om_0 \leadsto \C^n,
\]
where $\tilde q \in \N_+$,
is contained in $\B$. 
Since $h = unit \cdot \prod_{i=1}^r y_i^{m_i}$ for $m_i >0$ on $\cU$ (see \eqref{normalcrossings}), 
we have $|h| \lesssim |y_i|$ for all 
$1 \le i \le r$.
Let $\tilde \ps \in C^1(\ol \B)$ and consider the multi-valued function (cf.\ \eqref{varphiksecond})
\begin{equation*} 
  \vh :=\ul h^{\tilde \al} \tilde \ps\big(\ul y^{1/\tilde q}\big) : \Om_0 \leadsto \C;  
\end{equation*}
additionally we define
$\vh(x) := 0$ if $\ul h(x) = 0$.

\[
\xymatrix{
  \C^n=M_1  &&  V_0 \ar@{_{ (}->}[ll] \ar[rr] \ar[drr]|-{y_j = \frac{P_j}{h^{s}}} 
  && \cU \ar[d]^{\mathbf{y}} \ar@{^{ (}->}[rr] && M_\ell
  \\
  \R^m \ar[u]^a &&  \Om_0 \ar[u]^a \ar@{_{ (}->}[ll] \ar[rr]_{\ul y} && \B \ar@{^{ (}->}[rr] && \C^n
}
\]

\begin{remark}
  The functions $\ul h$ and $\ul P_i$ are defined on $\R^m$ and have compact support, 
  but their supports need not be contained in $\Om_0$.
\end{remark}

\begin{proposition} \label{prop:onechart} 
Let $k \in \N_+$ be a multiple of $\tilde q$ satisfying $k \ge   \max\big\{\frac{2s}{\tilde \al}, m\big\}$ for the numbers $\tilde q,s,\tilde \al$ associated with the chart $\cU$ 
and $\vh$ (see \eqref{varphiksecond}). 
Suppose that $a \in C^{k-1,1}_c(\R^m, \ol \B)$.
There exists a finite collection of $C^{k-1}$-hypersurfaces $E_i \subseteq \R^m$ 
and a parameterization $\ph$ of $\vh$ on $\Om_0$ such that $\ph \in SBV(\Om_0)$, 
$\ph $ is continuous on $\Om_0 \setminus \bigcup_i E_i$,
 and 
\begin{equation} \label{eq:boundchart} 
  \|\ph\|_{BV(\Om_0)} \le  C\, \big(|\Om_0| +1\big)\,  \Big(\|\ul h\|_{C^{k-1,1}}^{1/k} + \sum_{i=1}^n \|\ul P_i \|_{C^{k-1,1}}^{1/k}\Big),
\end{equation}
with $C>0$ depending only on
 $m$, $k$, $\tilde \al$, $\tilde q$, $s$, $\|\ul h\|_{L^\infty}$, $\|\tilde \ps\|_{C^1}$, and 
  $\|\ul P_j\|_{L^\infty}$ for $1 \le j \le r$.  
\end{proposition}

\begin{proof}
  This follows from \Cref{prop:local} which will be proved in the next section.
\end{proof}

\begin{convention} \label{convention}
  We will consistently write $\ph$ for a parameterization of the multi-valued function $\vh$.  Then $\ph$ is an ordered tuple of single-valued complex functions. 
\end{convention}

\Cref{prop:onechart}, \Cref{prop:extension}, 
and Formula \eqref{sums} give a parameterization of the roots by functions of bounded variation 
on the set corresponding to one chart.  
Our next step is to cut and paste such local roots. For this we use the following lemma.  

\begin{lemma} \label{cutglue}  
Let $u \in C^1(\ol \B)$ have values in the interval  $[b,c]$.
Suppose that $a \in C^{k-1,1}_c(\R^m, \ol \B)$, where $k \ge  \max\big\{\frac{s}{\tilde \al}, m\big\}$.
Consider the function $u \o \ul y : \Om_0 \to \R$. 
There is a constant $C>0$ which depends only on  $n$, $m$, $k$, 
$\|\ul h\|_{L^\infty}$, and $\|\nabla u\|_{L^\infty}$
such that for all small $\ep >0$ 
we have 
\begin{equation*}
  \Big|\Big\{r \in  [b,c] : \int_{(u\o \ul y)^{-1}(r)} |\ul h|^{\tilde \al}  \, d \cH^{m-1} 
    \ge \ep^{-1} C \, \Big(\|\ul h\|_{C^{k-1,1}}^{1/k} + \sum_{j=1}^n \|\ul P_j \|_{C^{k-1,1}}^{1/k}\Big) 
    \Big\}\Big| \le \ep.
\end{equation*}
\end{lemma}

\begin{proof}
  We have, on $\Om_0$, 
  \begin{multline*}
    |\ul h|^{\tilde \al}\,  |\nabla (u \o \ul y)|  
    \\ \lesssim  
    |\nabla u(\ul y)| \max_j \Big( |\ul y_j|^{1-1/k} |\ul h|^{\tilde \al-s/k} \frac{|\nabla \ul P_j|}{|\ul P_j|^{1-1/k}} 
    + s |\ul y_j| |\ul h|^{\tilde \al-1/k} \frac{|\nabla \ul h|}{|\ul h|^{1-1/k}} \Big).
  \end{multline*}
  By \Cref{GhisiGobbino2}, 
  \begin{align} \label{eq:est18}
    \int_{\Om_0} |\ul h|^{\tilde \al}\, \big|\nabla (u \o \ul y)\big| \, dx 
    \le  C\, \Big(\|\ul h\|_{C^{k-1,1}}^{1/k} + \sum_{j=1}^n \|\ul P_j \|_{C^{k-1,1}}^{1/k}\Big),
  \end{align} 
  for a constant $C$ as specified in the assertion.
  By the coarea formula (\Cref{changeofvariables}),
  \begin{equation*}
    \int_{\Om_0} |\ul h|^{\tilde \al}\big|\nabla (u \o \ul y)\big| \,dx 
    =  \int_b^c \int_{(u\o \ul y)^{-1}(r)} |\ul h|^{ \tilde \al} \, d \cH^{m-1} \, dr,
  \end{equation*}
 and the assertion follows as at the end of the proof of \Cref{prop:level}.
\end{proof}

\subsection{Choice of neighborhoods}\label{sec:neighborhoods}

 In what follows we fix a finite family of  extended chains  
\begin{align}\label{eqcover}
\cover = \{ \sfE_j \} =\{ (p_{j, \kk},f_{j, \kk}, h_{j,\kk}, P_{j,\kk,i}, s_{j,\kk} , r_{j,\kk},\mathcal U_{j, \kk}) \}
\end{align}
such that $\sigma_n ^{-1} (\overline \B) \subseteq \bigcup_{j=1}^N \mathcal U_{j,n}$. 
Such a family exists by compactness of $\sigma_n ^{-1} (\overline {\B})$.  Since $\si_{\kk,\kk-1}(\cU_{j,\kk}) \subseteq \cU_{j,\kk-1}$ for all $\kk$, we have  
\begin{align}\label{eq:topcover}
  \sigma_\kk ^{-1} (\overline \B) \subseteq \bigcup_j \mathcal U_{j,\kk}
\end{align}
for all levels $\ell$.
We will denote by $C(\cover, m)$ any constant which depends only on the family $\cover$ and $m$.

\subsection{Definition of $k(n,m)$} \label{ssec:kp}

Let $\sfE= (p_{\kk}, f_{\kk}, h_\kk, P_{\kk,i}, s_\kk , r_\kk, \mathcal U_\kk)$ be an extended chain
and let $\tilde q_\kk$ and $\tilde \al_\kk$ be as specified in \eqref{varphiksecond}. 
Let $k_\sfE \in \N_+$ be the smallest common multiple of the integers $\tilde q_\kk$ which satisfies 
\[
  k_\sfE \ge \max _\kk  \frac {2s_\kk} {\tilde \al_\kk}.
\]
Let $\cover = \{ \sfE_j \} =\{ (p_{j, \kk},f_{j, \kk}, h_{j,\kk}, P_{j,\kk,i}, s_{j,\kk} , r_{j,\kk},\mathcal U_{j, \kk}) \}$ be 
the finite family of  extended chains  fixed in \Cref{sec:neighborhoods}.   
Then we set    
\begin{align*}
  k = k(n,m)  := \max_{j} \big\{ k_{\sfE_j}, m \big\}.
\end{align*}

\subsection{Cutting and pasting} \label{sec:cutpaste}

Let $\kk$ be fixed. For each $j =1,\ldots, N$ (here $N$ is the number of extended chains in the family $\cover$) 
we have a chart $(\cU_{j,\kk},\mathbf{y}_{j,\kk})$ with a 
privileged system of coordinates, as in \Cref{sec:onechart}, and we 
consider the multi-valued functions  
\begin{align*}
\varphi_{j,\kk} = \ul h_{j,\kk}^{\tilde \al_{j,\kk}} \tilde \psi_{j,\kk} \big(\ul y_{j,\kk,1}^{1/\tilde q_{j,\kk}}, \ldots , 
\ul y_{j,\kk,r_{j,\kk}}^{1/\tilde q_{j,\kk}}, 
\ul y_{j,\kk,r_{j,\kk}+1}, \ldots , y_{j,\kk,n}\big) : \Om_{j,\kk,0}  \leadsto \C,
\end{align*}
where $\Om_{j,\kk,0} \subseteq \R^m$ is the open set defined in \eqref{eq:defOm0}. 
Then 
\begin{equation} \label{eq:Rmcover}
  \Om_{\kk,0}:=\R^m \setminus a^{-1}(V(\sD_\kk)) = \bigcup_{j=1}^N \Om_{j,\kk,0}.
\end{equation}
By \Cref{prop:onechart}, 
there is a parameterization $\ph_{j,\kk}$ of $\vh_{j,\kk}$ in $SBV(\Om_{j,\kk,0})$ 
such that 
\begin{equation*}
   \|\ph_{j,\kk}\|_{BV(\Om_{j,\kk,0})} \le  C(\cover ,m)\, \Big(\|\ul h_{j,\kk}\|_{C^{k-1,1}}^{1/k} 
   + \sum_{i=1}^n \|\ul P_{j,\kk,i} \|_{C^{k-1,1}}^{1/k}\Big).
\end{equation*}
(By \eqref{eq:boundchart}, the constant actually also depends on the Lebesgue measure of the support of $a$. 
This dependence only comes from the $L^1$-part in the $BV$-norm (cf.\ \Cref{prop:local}) and we will not write it.)

For each nonempty subset $J \subseteq \{1,\ldots,N\}$ set $\Om_{J,\kk,0} := \bigcap_{j \in J} \Om_{j,\kk,0}$.

\begin{lemma}
  We may shrink the sets $\Om_{j,\kk,0}$, $j = 1,\ldots,N$, in such a way that \eqref{eq:Rmcover} still holds  
  and $\ph_{j,\kk} \ind_{\Om_{J,\kk,0}} \in SBV(\Om_{\kk,0})$ whenever $j \in J$ and $J \subseteq \{1,\ldots,N\}$. 
  Moreover,
  \begin{equation*}
   \|\ph_{j,\kk}\ind_{\Om_{J,\kk,0}}\|_{BV(\Om_{\kk,0})} \le  C(\cover ,m)\, \max_{1 \le j \le N} \Big(\|\ul h_{j,\kk}\|_{C^{k-1,1}}^{1/k} 
   + \sum_{i=1}^n \|\ul P_{j,\kk,i} \|_{C^{k-1,1}}^{1/k}\Big).
\end{equation*}
\end{lemma}

\begin{proof}
We will use \Cref{cutglue} in each $\Om_{j,\ell,0}$ 
for the functions $u = |\cdot|$ and 
$u = |\mathbf{y}_{i,\kk} \o \mathbf{y}_{j,\kk}^{-1}|$ whenever $\cU_{i,\kk} \cap \cU_{i,\kk} \ne \emptyset$. 
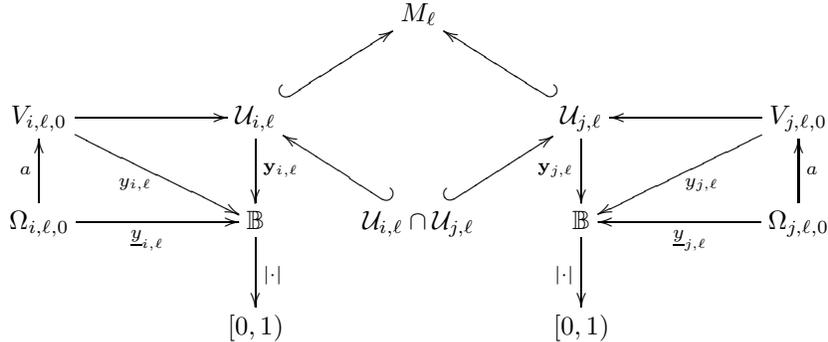
\[
\xymatrix{
  &&& M_\kk& & &
  \\
    V_{i,\kk,0}  \ar[rr]  \ar[drr]_{ y_{i,\kk}}
  && \cU_{i,\kk} \ar[d]^{\mathbf{y}_{i,\kk}} \ar@{^{ (}->}[ur]  && \cU_{j,\kk} \ar@{_{ (}->}[ul] 
  \ar[d]_{\mathbf{y}_{j,\kk}} & & V_{j,\kk,0} \ar[ll]  \ar[dll]^{ y_{j,\kk}}
  \\
     \Om_{i,\kk,0} \ar[u]^a  \ar[rr]_{\ul { y}_{i,\ell}} && \B  \ar[d]^{|\cdot|} & 
     \cU_{i,\kk} \cap \cU_{j,\kk} \ar@{^{ (}->}[ur] 
     \ar@{_{ (}->}[ul]  & \B  \ar[d]_{|\cdot|}&& \Om_{j,\kk,0} \ar[u]_a \ar[ll]^{\ul { y}_{j,\ell}}
  \\
  && [0,1) &&  [0,1) &&
}
\]
So we apply \Cref{cutglue} a finite number of times and hence obtain a real number $r<1$ close to $1$ such that 
\[
 \int_{(u\o \ul {y}_{j,\kk})^{-1}(r)} |\ul h_{j,\kk}|^{\tilde \al_{j,\kk}}  \, d \cH^{m-1} 
    \le C(\cover ,m)\, \Big(\|\ul h_{j,\kk}\|_{C^{k-1,1}}^{1/k} 
   + \sum_{i=1}^n \|\ul P_{j,\kk,i} \|_{C^{k-1,1}}^{1/k}\Big), 
\] 
for all $j$ and all $u$ as specified above. 
By Sard's theorem (\Cref{Sard}), we may also assume that for all $j$ and all $u$ as considered in the proof 
  the sets $(u\o \ul { y}_{j,\kk})^{-1}(r)$ are $C^{k-1}$-hypersurfaces.
Replacing the $\Om_{j,\kk,0}$ by the open subsets
\[
  \{x \in \R^m : |\ul {y}_{j,\kk}(x)| <r, \, \ul h_{j,\kk}(x) \ne 0 \} 
\]
the assertion of the lemma follows (cf.\ \cite[Theorem 3.77]{AFP00}).
\end{proof}

The map 
\begin{equation} \label{eq:phpara}
  \ph_\kk := \sum_{k=1}^N \frac{(-1)^{k+1}}{k} \sum_{|J| = k} \sum_{j \in J} \ph_{j,\kk} \ind_{\Om_{J,\kk,0}} 
\end{equation}
is a parameterization of the multi-valued function
\begin{equation*}
  \vh_\kk := \sum_{k=1}^N \frac{(-1)^{k+1}}{k} \sum_{|J| = k} \sum_{j \in J} \vh_{j,\kk} \ind_{\Om_{J,\kk,0}} 
\end{equation*}
which belongs to $SBV(\Om_{\ell,0})$ and verifies
\begin{equation} \label{eq:boundbeforeextension}
   \|\ph_\kk\|_{BV(\Om_{\kk,0})} \le  C(\cover ,m)\, \max_{1 \le j \le N} \Big(\|\ul h_{j,\kk}\|_{C^{k-1,1}}^{1/k} 
   + \sum_{i=1}^n \|\ul P_{j,\kk,i} \|_{C^{k-1,1}}^{1/k}\Big).
\end{equation}
In the definition of $\vh_\kk$ all $\vh_{j,\kk}$, $j\in J$, are equal as multi-valued functions on $\Om_{J,\kk,0}$ 
and hence their arithmetic mean should be interpreted as the same multi-valued function.

Let us extend $\ph_\kk$ by $0$ to the set $a^{-1}(V(\sD_\kk)) = \R^m \setminus \Om_{\kk,0}$. 
Then $\ph_\kk$ is continuous at all points of $\R^m \setminus \Om_{\kk,0}$; furthermore, $\ph_\kk$
has compact support. 
By \Cref{prop:extension}, we may conclude that $\ph_\kk$ has bounded variation on $\R^m$ and 
$\|\ph_\kk\|_{BV(\R^m)}=\|\ph_\kk\|_{BV(\Om_{\kk,0})}$ is bounded by \eqref{eq:boundbeforeextension}. 
We have $S_{\ph_{\kk}} = J_{\ph_{\kk}}$ and the Cantor part  of $D\ph_\kk$ vanishes in $\Om_{\kk,0}$.
  The Cantor part  of $D \ph_\kk$ also vanishes in $\R^m \setminus \Om_{\ell,0}$, 
  by \cite[Proposition 3.92]{AFP00}, since $\ph_\kk$ vanishes and 
    is continuous on this set.   Thus $\ph_\kk \in SBV(\R^m)$.

Now we can use \Cref{roots}, that gives a parameterization 
$\la : \R^m \to \C^n$  
 of the roots of $P_a$, where $a \in C^{k-1,1}_c(\R^m, \ol \B)$, 
as a finite sum of the above constructed  $\ph_\kk$ 
and therefore belongs to 
$SBV(\R^m)$ and satisfies 
\begin{equation} \label{eq:laB}
   \|\la\|_{BV(\R^m)} \le  C(\cover ,m)\, \max_{1 \le \kk \le n}\max_{1 \le j \le N} \Big(\|\ul h_{j,\kk}\|_{C^{k-1,1}}^{1/k} 
   + \sum_{i=1}^n \|\ul P_{j,\kk,i} \|_{C^{k-1,1}}^{1/k}\Big).
\end{equation}

\begin{remark}
  Actually, to get 
  an everywhere defined  
  parameterization $\la$ of the roots 
  of $P_a$ a modification of the parameterizations $\ph_\kk$ 
  defined in \eqref{eq:phpara} 
  is necessary on their discontinuity sets which are $\cL^m$-negligible.
 
 Their values on these sets can be chosen arbitrarily, with the only condition that they should satisfy 
 $P_a(Z) = \prod_{j=1}^n (Z-\la_j)$, since the 
 membership in $SBV(\R^m)$ and the bound for $\|\ph_\kk\|_{BV(\R^m)}$ are not affected by this modification.  
\end{remark}

\subsection{General Case}\label{subs:GeneralCase}

Let $\Om \subseteq \R^m$ be a bounded Lipschitz domain and
let $a \in C^{k-1,1}(\overline \Om,\C^n)$.
By Whitney's extension theorem, 
  $a$ admits a $C^{k-1,1}$-extension $\hat a$ to $\R^m$ 
  with compact support in the open $1$-neighborhood $\Om_1$ of $\Om$
  such that  
  \begin{equation} \label{eq:whitney4}
     \|\hat a\|_{C^{k-1,1}(\overline \Om_1)}\le C(m,k,\Om) \,
     \|a\|_{C^{k-1,1}(\overline \Om)};
  \end{equation}
  see \Cref{sec:prep2} for more details.
  Let us first prove \Cref{thm:main} for $\hat a$.

We may suppose that the discriminant of $P_{\hat a(x)}$ is not identically equal to zero.    

In general, the image $\hat a (\R^m)$ is not necessarily contained in the closed unit ball 
$ \ol \B\subseteq \C^n$.   To reduce to this case we use an $\R_+$ action 
on the 
coefficient vector $\hat a\in \C^n$.  
 For $\eta>0$ and $\hat a\in \C^n$ we define $\eta * \hat a\in \C^n$ by $(\eta * \hat a)_i = \eta^i \hat a_i$.  Then $\lambda$ 
 is a root of $P_{\hat a}$ if and only if $\eta  \lambda$ is a root of $P_{\eta *  \hat a}$.   
 
 Fix $\rho \ge \max\{1, \|\hat a\|_{L^\infty} \}$.
 Then $ \| \rho ^{-1}* \hat a\|_{L^\infty} \le 1$.  
 By \eqref{eq:laB} applied to $P_{\rho^{-1} *  \hat a}$
 the roots of $P_{\hat a}$ admit a parameterization $\la : \R^m \to \C^n$ in 
  $SBV(\R^m)$ such that
 \begin{equation*} 
   \|\la\|_{BV(\R^m)} \le \rho\, C(n,m)\, \max_{1 \le \kk \le n}\max_{1 \le j \le N} 
   \Big(\|\underaccent{\tilde} h_{j,\kk}\|_{C^{k-1,1}}^{1/k} 
   + \sum_{i=1}^n \|\underaccent{\tilde} P_{j,\kk,i} \|_{C^{k-1,1}}^{1/k}\Big),
\end{equation*} 
 where for a polynomial $g\in \C[\hat a]$ 
 we set
 $\underaccent{\tilde} {g} (x) := g(\rho ^{-1}*  \hat a(x))$. 
 (The dependence of the constant on the cover $\cover$ is subsumed under its dependence on $n$.)

By \Cref{sec:neighborhoods} and by the regularity of the composition of $C^{k-1,1}$-maps (cf.\ \cite[Theorem 4.3]{LlaveObaya99}),
we conclude that 
\begin{align*}
  \|\la\|_{BV(\R^m)} &\le 
  C(n,m)\,   \max_{1 \le \kk \le n}\max_{1 \le j \le N} 
   \Big(\| h_{j,\kk}\|_{C^{k-1,1}(\ol \B)}^{1/k} 
   + \sum_{i=1}^n \| P_{j,\kk,i} \|_{C^{k-1,1}(\ol \B)}^{1/k}\Big) 
   \\
   &\hspace{2cm}
   \times \max\{1, \|\hat a\|_{L^\infty}\} \big(1 + \|\hat a\|_{C^{k-1,1}}\big)
   \\
   &=
  \tilde C(n,m)\, \max\{1, \|\hat a\|_{L^\infty}\} \big(1 + \|\hat a\|_{C^{k-1,1}}\big)  
 \end{align*} 
 for a different constant $\tilde C(n,m)$.

  The proof of \Cref{thm:main} for coefficients $a \in C^{k-1,1}(\overline \Om,\C^n)$ defined on an arbitrary 
  bounded Lipschitz domain now follows
  in view of \eqref{eq:whitney4}:
  we obtain that  
  the roots of $P_{a}$ admit a parameterization $\la : \Om \to \C^n$ in 
  $SBV(\Om)$ such that
  \begin{align*}
  \|\la\|_{BV(\Om)} 
   &\le
  C(n,m,\Om)\, \max\{1, \|a\|_{L^\infty(\Om)}\} \big(1 + \|a\|_{C^{k-1,1}(\ol \Om)}\big).  
 \end{align*}
 The statement about the discontinuity set of $\la$ follows from the construction.
 To finish the proof of \Cref{thm:main} we just need to complete the proof of \Cref{prop:onechart} which  
 is done in the next section.

\section{A local computation} \label{localproof}

In this section we complete the proof of \Cref{prop:onechart} by showing a slightly more general   
\Cref{prop:local}. 
First we fix the setup and the notation  that we use throughout this section.

\begin{setup} \label{setup}
Let $m$ and $n$ be integers $\ge 2$.
Let $q, s \in \N_+$ and $\al \in q^{-1}\N_+$.
Let   
\begin{equation} \label{eq:setup1}
  \text{$k \in \N_+$ be a multiple of $q$ satisfying } 
  k \ge   \max\Big\{\frac{2s}{\al}, m\Big\} .  
\end{equation}
Then $k/q \in \N_+$ and $k \al \in \N_+$.
Let $h,P_j \in C^{k-1,1}_c(\R^m)$, $j =1,\ldots,n$, and set $y_j := P_j/h^s$.
Put $y:=(y_1,\ldots,y_n)$ and consider the bounded open subset of $\R^m$,  
\begin{align*}
  \Om_0 &:= \{x \in \R^m : |y(x)| < 1,\, h(x) \ne 0\}.  
\end{align*}
We assume that, on $\Om_0$,
\begin{equation} \label{eq:setup5}
  |h| \lesssim |y_j| \quad \text{ for all } j =1,\dots,r.   
\end{equation}
This implies that no $P_j$, $j =1,\dots,r$, vanishes on $\Om_0$.
Let $y^{1/q}$ denote the multi-valued function 
\begin{equation} \label{eq:setup3}
 y^{1/q}:=\big(y_1^{1/q},\ldots,y_r^{1/q}, y_{r+1},\ldots, y_n\big) : \Om_0 \leadsto \C^n; 
\end{equation} 
it takes values in the open unit ball $\B = \{z \in \C^n : |z|<1\}$ in $\C^n$.
Let $\ps \in C^1(\overline \B)$ and
consider the multi-valued function
\begin{equation} \label{eq:setup4}
  \vh :=h^\al \ps\big(y^{1/q}\big) = 
  h^\al \ps\big(y_1^{1/q},\ldots,y_r^{1/q}, y_{r+1},\ldots, y_n\big) : \Om_0 \leadsto \C;  
\end{equation}
additionally we define
$\vh(x) := 0$ if $h(x)=0$.  
\end{setup}

\subsection{Continuous parameterization of $\vh$}
 
The multi-valued function $\vh$ in 
\eqref{eq:setup4}  has a continuous parameterization in the complement of a finite union of $C^{k-1}$-hypersurfaces.  

\begin{lemma} 
For each $1 \le j \le r$, there is a $C^{k-1}$-hypersurfaces $E_j \subseteq \Om_0$ 
such that the multi-valued function $y_j^{1/q}$ has a continuous parameterization on $\Om_0 \setminus E_j$.
There exists a $C^{k-1}$-hypersurface $E_0 \subseteq \R^m$ such that $h^\al$ has a continuous parameterization
on $\R^m \setminus E_0$. 
There is a continuous parameterization of $\vh$ on $\Om_0 \setminus \bigcup_{j=0}^r E_j$. 
Every parameterization of $\ph$ extends continuously to the zero set of $h$. 
(Here \eqref{eq:setup1} can be replaced by the weaker condition $k \ge m$.)
\end{lemma}

\begin{proof}
  Each $y_j = P_j/h^s$ is of class $C^{k-1,1}$ on $\Om_0$. 
  By \Cref{discont},
  there exists a $C^{k-1}$-hypersurface $E_j \subseteq \Om_0$  
  such that the multi-valued function $y_j^{1/q}$, where $1 \le j\le r$, 
  has a continuous parameterization on $\Om_0 \setminus E_j$.
  (If $r<  j \le n$ then $y_j$ is a single-valued continuous function on $\Om_0$.)  
  Likewise the multi-valued function $h^{\al}$ has a continuous parameterization on $\R^m \setminus E_0$, 
  where $E_0$ is a $C^{k-1}$-hypersurface in $\R^m$.
  Now any parameterization of $\vh$ extends continuously to  
  $h^{-1}(0)$ because $\ps(y^{1/q})$ is bounded. 
\end{proof}

\subsection{Parameterization of $\vh$ with bounded variation}

Now we are ready to show that $\vh$ admits a parameterization of bounded variation.

\begin{proposition} \label{prop:local} 
There exists a finite collection of $C^{k-1}$-hypersurfaces $E_j \subseteq \R^m$ 
and a parameterization $\ph$ of $\vh$ on $\Om_0$ such that $\ph \in SBV(\Om_0)$, 
$\ph$ is continuous on $\Om_0 \setminus \bigcup_j E_j$,   
\begin{equation*}
  |D \ph|(\Om_0) \le  C\, \Big(\|h\|_{C^{k-1,1}}^{1/k} + \sum_{j=1}^n\|P_j \|_{C^{k-1,1}}^{1/k}\Big),
\end{equation*}
and 
\begin{equation} \label{phBVnorm}
  \|\ph\|_{BV(\Om_0)}  \le  C\, \big(|\Om_0| +1 \big)\, \Big(\|h\|_{C^{k-1,1}}^{1/k} + \sum_{j=1}^n\|P_j \|_{C^{k-1,1}}^{1/k}\Big),
\end{equation}
with $C>0$ depending only on
 $m$, $k$, $\al$, $q$, $s$, $\|h\|_{L^\infty}$, $\|\ps\|_{C^1}$, and 
  $\|P_j\|_{L^\infty}$ for $1 \le j \le r$.
We have $S_\ph = J_\ph \subseteq \bigcup_{j} E_j$ and the Cantor part $D\ph$ vanishes in $\Om_0$.
\end{proposition}

\begin{proof}
  First we choose a parameterization $\ph$ of $\vh$ and its discontinuity set:
  By \Cref{BVradicals}, 
there is a closed $C^{k-1}$-hypersurface $E_0$ in $\R^m \setminus h^{-1}(0)$ and
 $h^{1/k}$ admits a parameterization in $SBV(\R^m)$ 
  (recall that $h$ has compact support), 
  again denoted by $h^{1/k}$, which is continuous on $\R^m \setminus E_0$ and satisfies 
  \begin{equation*}
    S_{h^{1/k} } = J_{h^{1/k} } = E_0 \quad \text{ and } \quad D^c (h^{1/k} ) = 0
  \end{equation*}
  and 
  \begin{equation} \label{eq:levelh0}
    \int_{E_0} |h|^{1/k}\, d\cH^{m-1} \le C(m,k)\, \|h\|_{C^{k-1,1}}^{1/k}.
  \end{equation}
  Analogously, for each $1 \le j \le r$, 
  there is a closed $C^{k-1}$-hypersurface $E_j$ in $\R^m \setminus P_j^{-1}(0)$ and
 $P_j^{1/k}$ admits a parameterization in $SBV(\R^m)$, 
  again denoted by $P_j^{1/k}$, which is continuous on $\R^m \setminus E_j$ and satisfies 
  \begin{equation*}
    S_{P_j^{1/k} } = J_{P_j^{1/k} } = E_j \quad \text{ and } \quad D^c (P_j^{1/k} ) = 0
  \end{equation*}
  and 
  \begin{equation} \label{eq:levelPj}
    \int_{E_j} |P_j|^{1/k}\, d\cH^{m-1} \le C(m,k)\, \|P_j\|_{C^{k-1,1}}^{1/k}.
  \end{equation}

  Since $k \al \in \N_+$ (cf.\ \eqref{eq:setup1}), 
  \[
    (h^{1/k})^{k\al} = \exp\big(\tfrac{1}{k} \log h\big)^{k\al} =  \exp\big(\tfrac{k\al}{k} \log h\big) = 
    \exp\big(\al \log h\big) = h^\al 
  \]  
  is a parameterization of the multi-valued function $h^\al$ in $SBV(\R^m)$ which is continuous on 
  $\R^m \setminus E_0$ and satisfies 
  \begin{equation*}
    S_{h^{\al} } = J_{h^{\al} } = E_0 \quad \text{ and } \quad D^c (h^{\al} ) = 0.
  \end{equation*}
  Similarly, since $k/q \in \N_+$, $h^{1/q} = (h^{1/k})^{k/q}$ is a parameterization of the multi-valued function $h^{1/q}$ 
  in $SBV(\R^m)$
  which is continuous on 
  $\R^m \setminus E_0$ and satisfies 
  \begin{equation*}
    S_{h^{1/q} } = J_{h^{1/q} } = E_0 \quad \text{ and } \quad D^c (h^{1/q} ) = 0.
  \end{equation*}
  By the same reason $P_j^{1/q} = (P_j^{1/k})^{k/q}$, for $1 \le j \le r$, is a parameterization of the multi-valued function $P_j^{1/q}$ 
  in $SBV(\R^m)$
  which is continuous on 
  $\R^m \setminus E_j$ and satisfies 
  \begin{equation*}
    S_{P_j^{1/q} } = J_{P_j^{1/q} } =  E_j \quad \text{ and } \quad D^c (P_j^{1/q} ) = 0.
  \end{equation*}

  For $1 \le j \le r$, the identity $y_j^{1/q} = P_j^{1/q}/ h^{s/q}$ yields a parameterization in $SBV_{\on{loc}}(\Om_0)$ 
  of the multi-valued function $y_j^{1/q}$ which is 
  continuous on 
  $\Om_0 \setminus (E_0 \cup E_j)$ and satisfies 
  \begin{equation*}
    S_{y_j^{1/q} } = J_{y_j^{1/q} } \subseteq \Om_0 \cap (E_0 \cup E_j) \quad \text{ and } \quad D^c (y_j^{1/q} ) = 0.
  \end{equation*}
  Then, by Theorem \ref{thm:chainrule} and \eqref{productrule}  
  (note that $\ps$ admits a Lipschitz $C^1$-extension to $\R^{2n}$), 
  there is a parameterization $\ph$ of $\vh$ in $SBV_{\on{loc}}(\Om_0)$ 
  which is continuous in $\Om_0 \setminus E$, where $E := \bigcup_{j=0}^r E_j$, and satisfies 
  \begin{equation*}
    S_{\ph} = J_{\ph} \subseteq \Om_0 \cap E \quad \text{ and } \quad D^c \ph = 0.
  \end{equation*}

  The rest of the proof is devoted to showing that $\ph$ has bounded variation on $\Om_0$ (not just locally); 
  for this we use the chain rule \eqref{chainrule} and the product rule \eqref{productrule}. 
  Note that on $h^{1/k}$ coincides with its precise representative on $\R^m \setminus E_0$; 
  similarly for $P_j^{1/k}$, etc.
  In the following computations we let $q = 1$ if $r+1 \le j \le n$. First of all,
  \begin{align*}
    \tD \ph  &= h^\al \tD( \ps(y^{1/q}) ) + \ps(y^{1/q}) \tD(h^\al) 
    \\
    &= h^\al \, \nabla \ps(y^{1/q}) \tD(y^{1/q}) + \ps(y^{1/q}) \tD(h^\al).
  \end{align*}
  Since $k\al \in \N_+$,
  the map $z \mapsto z^{k\al}$ is $C^1$ and hence, by the chain rule \eqref{chainrule},
  \begin{align*}
    \tD (h^{\al}) &= \tD \big((h^{1/k})^{k\al}\big)   
    = k\al\, (h^{1/k})^{k\al-1}\, \tD(h^{1/k}) = k\al\,   h^{\al-1/k}\, \tD(h^{1/k}).
  \end{align*}
  Analogously, we find 
  \begin{align*}
    \tD (h^{1/q}) &= \frac{k}{q}   h^{1/q-1/k}\, \tD(h^{1/k}) \quad \text{and} \quad
    \tD (P_j^{1/q}) =   \frac{k}{q} P_j^{1/q-1/k}\, \tD(P_j^{1/k}).
  \end{align*}
  Altogether,
  \begin{align*}
    \tD \ph  &= h^\al \, \nabla \ps(y^{1/q}) \Big(P_j^{1/q} \tD(h^{-s/q}) + h^{-s/q} \tD(P_j^{1/q}) \Big)_j + \ps(y^{1/q}) \tD(h^\al) 
    \\
    &= h^\al \, \nabla \ps(y^{1/q}) \Big(-\frac{ks}{q} P_j^{1/q} h^{-s/q-1/k} \tD(h^{1/k}) 
    +  \frac{k}{q} h^{-s/q} P_j^{1/q-1/k}\, \tD(P_j^{1/k}) \Big)_j 
    \\
    &\quad
    + k \al\, \ps(y^{1/q}) h^{\al-1/k}\tD(h^{1/k})
    \\
    &=  \nabla \ps(y^{1/q}) \Big(-\frac{ks}{q} y_j^{1/q} h^{\al-1/k} \tD(h^{1/k}) 
    +  \frac{k}{q} y_j^{1/q-1/k} h^{\al-s/k} \, \tD(P_j^{1/k}) \Big)_j 
    \\
    &\quad
    + k \al\, \ps(y^{1/q}) h^{\al-1/k}\tD(h^{1/k}).
  \end{align*}   
  Since $\ps \in C^1(\overline \B)$ and $k$  
  satisfies \eqref{eq:setup1}, we have on $\Om_0$,
  \begin{align*} 
    |\tD \ph|  
        & \le  C\, \Big( |\tD (h^{1/k})|+ \sum_{j=1}^n |\tD(P_j^{1/k})| \Big), 
  \end{align*} 
  where $C$ is a constant only depending on $k$, $\al$, $q$, $s$, $\|h\|_{L^\infty}$, and $\|\ps\|_{C^1}$ 
    (recall that $|y_j|<1$).

  On the other hand for the jump part, \eqref{chainrule} and \eqref{productrule} imply
  \begin{align*}
    D^j \ph = \Pi \otimes \nu_{(h^\al,\ps(y^{1/q}))}\, \cH^{m-1} \restr J_{(h^\al,\ps(y^{1/q}))},
  \end{align*}
  where 
  \begin{align*}
    \Pi (x) = 
    \begin{cases}
      \big((h^\al)^+(x) - (h^\al)^-(x)\big)\, \ps(y^{1/q}(x)) & \text{ if } x \in J_{h^\al} \setminus J_{\ps(y^{1/q})}, \\
      h^\al(x) \, \big(\ps(y^{1/q})^+(x) -  \ps(y^{1/q})^-(x)\big) & \text{ if } x \in  J_{\ps(y^{1/q})} \setminus J_{h^\al}, \\
      (h^\al)^+(x)\,\ps(y^{1/q})^+(x) - (h^\al)^-(x)\,\ps(y^{1/q})^-(x)  & \text{ if } x \in J_{h^\al} \cap J_{\ps(y^{1/q})}.
    \end{cases}
  \end{align*} 
  Thus 
  \begin{align*}
    |D^j \ph| &\le C\, \sum_{j=0}^r |h|^\al \,  \cH^{m-1} \restr (\Om_0 \cap E) 
  \end{align*}
  for some constant $C$ depending only on $\|\ps\|_{L^\infty}$.  
  We may conclude that $\ph \in BV(\Om_0)$ and 
  \begin{align*} 
    |D \ph|(\Om_0) &\le |\tD \ph|(\Om_0) + |D^j\ph|(\Om_0) 
    \notag \\
    &\le C\, \Big( |\tD (h^{1/k})|(\Om_0)+ \sum_{j=1}^n |\tD(P_j^{1/k})|(\Om_0)  
      +   \sum_{j=0}^r \int_{\Om_0 \cap E_j} |h|^\al \,  d\cH^{m-1} \Big).
  \end{align*}
  By \Cref{GhisiGobbino2} (or \Cref{optimal2}), \eqref{eq:levelh0}, and \eqref{eq:levelPj} (using  
  $|h|^\al =|h|^{1/k} |h|^{\al-1/k}$ and $|h|^{\al} \lesssim |P_j|^{\al/(s+1)} =|P_j|^{1/k} |P_j|^{\al/(s+1) - 1/k}$, 
  by \eqref{eq:setup5}, where the second factors are bounded in both cases, by \eqref{eq:setup1}),
  we obtain 
   \begin{equation*} 
    |D \ph|(\Om_0) \le 
    C\, \Big(\|h\|_{C^{k-1,1}}^{1/k} + \sum_{j=1}^n \|P_j \|_{C^{k-1,1}}^{1/k}\Big),  
  \end{equation*} 
  for a constant $C>0$ which depends only on 
   $m$, $k$, $\al$, $q$, $s$, $\|h\|_{L^\infty}$, $\|\ps\|_{C^1}$, and 
  $\|P_j\|_{L^\infty}$ for $1 \le j \le r$.
  Then \eqref{phBVnorm} follows, 
  since $\|\ph\|_{L^1(\Om_0)} \le |\Om_0| \, \|h\|_{L^\infty}^\al \|\ps\|_{L^\infty}$.  
\end{proof}

\appendix

\section{Sobolev regularity of continuous roots} \label{appendix}

The next theorem is a refinement of \cite[Theorem 2]{ParusinskiRainer15} in the case that $\Om$ is a Lipschitz domain.

\begin{theorem}  \label{optimal2}
  Let $\Om \subseteq \R^m$ be a bounded Lipschitz domain.
  Let $P_a$  
  be a monic polynomial \eqref{polynomials} with coefficients $a_j \in C^{n-1,1}(\overline \Om)$, $j = 1,\ldots,n$.   
  Let $\la \in C^0(V)$ be a root of $P_a$ on an open subset $V \subseteq \Om$. 
  Then $\la$ belongs to the Sobolev space $W^{1,p}(V)$ for every $1 \le p < \frac{n}{n-1}$. 
  The distributional gradient $\nabla \la$ satisfies  
  \begin{equation} \label{A:multbound3} 
   \|\nabla \la \|_{L^p(V)}  \le  C(m,n,p,\Om) \max_{1 \le j \le n} \|a_j\|^{1/j}_{C^{n-1,1}(\overline \Om)}.
  \end{equation}  
\end{theorem}

\begin{proof}
  By \cite[Theorem 1]{ParusinskiRainer15}, $\la$ is absolutely continuous along affine lines parallel to 
  the coordinate axes (restricted to $V$). 
  So $\la$ possesses the partial derivatives $\p_i \la$, $i=1,\ldots,m$, which 
  are defined almost everywhere and are measurable. 
  
  Set $x=(t,y)$, where $t=x_1$, $y=(x_2,\ldots,x_m)$, 
  and let $V_1$ be the orthogonal projection of $V$ on the hyperplane $\{x_1=0\}$.
  For each $y \in V_1$ we denote by $V^y := \{t \in \R : (t,y) \in V\}$ the corresponding section of $V$.

  Let $\la^y_j$, $j=1,\ldots,n$, be a continuous system of the roots of $P_a(\cdot,y)$ on 
  $V^y$ such that $\la(\cdot,y) = \la^y_1$; it exists 
  since $\la(\cdot,y)$ can be completed to a continuous system of the roots of $P_a(\cdot,y)$ on each connected component 
  of $V^y$ by  \cite[Lemma 6.17]{RainerN}. 
  Our goal is to bound  
  \[
    \|\p_t\la(\cdot,y)\|_{L^p(V^y)} = \|(\la^y_1)'\|_{L^p(V^y)}
  \]
  uniformly with respect to $y \in V_1$.
  
  Let $R= I_1 \times \cdots \times I_m \subseteq \R^m$ be an open box containing $\Om$ and such that 
  $|I_i| \le \on{diam}(\Om)$ for all $ i =1,\ldots,m$.
  By Whitney's extension theorem (cf.\ \Cref{sec:prep2}), 
  the coefficients $a_j$ of $P_a$ admit a $C^{n-1,1}$-extension $\hat a_j$ to $\R^m$
  such that 
  \begin{equation} \label{Aeq:whitney}
    \max_{1 \le j \le n} \|\hat a_j\|^{1/j}_{C^{n-1,1}(\overline R)} \le C(m,n,\Om)\, 
    \max_{1 \le j \le n} \|a_j\|^{1/j}_{C^{n-1,1}(\overline \Om)}.
  \end{equation}

  Let $\cC^y$ denote the set of connected components $J$ of the open subset $V^y \subseteq \R$. 
  For each $J \in \cC^y$ we extend the system of roots $\la^y_j|_J$, $j=1,\ldots,n$, continuously to $I_1$, i.e., 
  we choose continuous functions $\la^{y,J}_j$, $j = 1,\ldots, n$, on $I_1$ such that $\la^{y,J}_j|_J = \la^y_j|_J$ 
  for all $j$ and 
  \begin{gather*}
      P_{\hat a}(t,y)(Z) = \prod_{j=1}^n (Z-\la^{y,J}_j(t)), \quad t \in I_1.
  \end{gather*} 
  This is possible since $\la^y_j|_J$ has a continuous extension to the endpoints of the (bounded) interval $J$, by 
  \cite[Lemma 4.3]{KLMR05}, and can then be extended on the left and on the right of $J$ by a continuous system 
  of the roots of $P_{\hat a}(\cdot,y)$ after suitable permutations.

  By \cite[Theorem 1]{ParusinskiRainer15}, for each $y \in V_1$, $J \in \cC^y$, and $j=1,\ldots,n$, 
  the function $\la^{y,J}_j$ is absolutely continuous on $I_1$ and $(\la^{y,J}_j)' \in L^p(I_1)$, for $1 \le p < n/(n-1)$, with 
  \begin{equation} \label{Aeq:ub}
     \|(\la^{y,J}_j)'\|_{L^p(I_1)} \le C(n,p,|I_1|) \, \max_{1 \le i \le n} \|\hat a_i\|^{1/i}_{C^{n-1,1}(\overline R)}.
  \end{equation}  

  Let $J,J_0 \in \cC^y$ be arbitrary. 
  By \cite[Lemma~3.6]{ParusinskiRainerAC}, $(\la^y_j)'$ as well as $(\la^{y,J_0}_j)'$ belong to $L^p(J)$ and 
  we have  
  \begin{equation*}
    \sum_{j=1}^n \|(\la^y_j)'\|_{L^p(J)}^p 
    = \sum_{j=1}^n \|(\la^{y,J}_j)'\|_{L^p(J)}^p = \sum_{j=1}^n \|(\la^{y,J_0}_j)'\|_{L^p(J)}^p.  
  \end{equation*}
  Thus,  
  \begin{align*}
    \sum_{j=1}^n \|(\la^y_j)'\|_{L^p(V^y)}^p 
    &= \sum_{J \in \cC^y} \sum_{j=1}^n \|(\la^y_j)'\|_{L^p(J)}^p 
    = \sum_{J \in \cC^y} \sum_{j=1}^n \|(\la^{y,J_0}_j)'\|_{L^p(J)}^p \\ 
    &= \sum_{j=1}^n \|(\la^{y,J_0}_j)'\|_{L^p(V^y)}^p 
    \le \sum_{j=1}^n \|(\la^{y,J_0}_j)'\|_{L^p(I_1)}^p.
  \end{align*}
  In particular, by \eqref{Aeq:ub},
  \begin{align*}
    \|\p_t\la(\cdot,y)\|_{L^p(V^y)} = \|(\la^y_1)'\|_{L^p(V^y)} 
    \le C(n,p,|I_1|) \, \max_{1 \le i \le n} \|\hat a_i\|^{1/i}_{C^{n-1,1}(\overline R)},
  \end{align*}
  and so, by Fubini's theorem, 
  \begin{align*} 
      \int_{V} |\p_1 \la(x)|^p\, dx &= \int_{V_1} \int_{V^y} |\p_1 \la(t,y)|^p\, dt\, dy 
      \\
      &\le \Big(C(n,p,|I_1|) \, \max_{1 \le i \le n} \|\hat a_i\|^{1/i}_{C^{n-1,1}(\overline R)} \Big)^p \int_{V_1} \, dy.
    \end{align*}  
  Thus, thanks to $|I_1| \le \on{diam}(\Om)$,
  \begin{equation*}
    \|\p_1 \la \|_{L^p(V)} \le C(n,p,\on{diam}(\Om)) \, \max_{1 \le i \le n} \|\hat a_i\|^{1/i}_{C^{n-1,1}(\overline R)}.   
  \end{equation*} 
  In view of \eqref{Aeq:whitney} this implies \eqref{A:multbound3}, since the other partial derivatives $\p_i \la$, $i \ge 2$, 
  are treated analogously. 
\end{proof}

\def\cprime{$'$}
\providecommand{\bysame}{\leavevmode\hbox to3em{\hrulefill}\thinspace}
\providecommand{\MR}{\relax\ifhmode\unskip\space\fi MR }
\providecommand{\MRhref}[2]{%
  \href{http://www.ams.org/mathscinet-getitem?mr=#1}{#2}
}
\providecommand{\href}[2]{#2}

\end{document}